\documentclass[reqno,twoside]{amsart}
\usepackage{amsfonts,amssymb,amsmath}
\usepackage{mathrsfs}
\usepackage{enumerate}
\allowdisplaybreaks
\newtheorem{theorem}{Theorem}[section]
\newtheorem{lemma}[theorem]{Lemma}
\newtheorem{proposition}[theorem]{Proposition}
\newtheorem{corollary}[theorem]{Corollary}

\newtheorem{remark}[theorem]{Remark}

\newtheorem{hyp}[theorem]{Hypothesis}

\newcommand{\R}{{\mathbb R}}
\newcommand{\Z}{{\mathbb Z}}
\newcommand{\N}{{\mathbb N}}

\newcommand{\U}{{\mathbb T}}

\newcommand{\hs}[1]{\hskip -#1pt}
\newcommand{\tr}{\mathrm{Tr}}
\newcommand{\one}{1\!\!\!\;\mathrm{l}}

 \newcommand{\supp}{{\rm supp\ }}
 \newcommand{\loc}{\rm loc}

\newcommand{\A}{\mathcal{A}}

\newcommand{\G}{\mathcal{G}}

\newcommand{\T}{\mathcal{T}}

\newcommand{\eps}{\varepsilon}

\numberwithin{equation}{section}
\begin{document}

\title[Asymptotic behavior in time periodic parabolic problems]{Asymptotic behavior
in time periodic parabolic problems with unbounded coefficients}
\author[L. Lorenzi, A. Lunardi, A. Zamboni]{Luca Lorenzi, Alessandra Lunardi, Alessandro Zamboni}

\address{Dipartimento di Matematica\\ Viale G.P. Usberti, 53/A \\
I-43100 Parma\\ Italia}
\email{luca.lorenzi@unipr.it}
\email{alessandra.lunardi@unipr.it}
\email{zambo1903@virgilio.it}


\begin{abstract}
We study asymptotic behavior in a class of non-autonomous second order parabolic equations with time periodic unbounded coefficients in $\R\times \R^d$. Our results generalize and improve  asymptotic behavior results for Markov semigroups having an invariant measure.
We also study spectral properties of the realization of the parabolic operator $u\mapsto \A(t) u - u_t$ in suitable $L^p$ spaces.
\end{abstract}

\maketitle

\section{Introduction}

We consider linear second-order differential operators,
\begin{align}
\label{A(t)}
(\A(t)\varphi) (x) &=\sum_{i,j=1}^d
q_{ij}(t,x)D_{ij}\varphi (x) + \sum_{i=1}^d b_i(t,x) D_i\varphi
(x)\notag\\
&= \tr \left( Q(t,x) D^2 \varphi (x)\right) + \langle b(t,x),
\nabla \varphi (x) \rangle,
\end{align}
with smooth enough coefficients defined in $ \R^{1+d}$,   satisfying the uniform ellipticity assumption
\begin{equation}
\label{ellipticity}
\sum_{i,j=1}^d q_{ij}(t,x) \xi_i\xi_j \geq \eta_0|\xi|^2, \quad (t,x)\in  \R^{1+d}, \;\xi\in \R^d.
\end{equation}
Under general assumptions, a Markov evolution operator $P(t,s)$
associated to the family $\{\A (t)\}$ has been constructed and
studied in \cite{KLL}. For every continuous and bounded $\varphi$
and for any $s\in \R$, the function $(t,x)\mapsto P(t,s)\varphi(x)$
is the unique bounded classical solution $u$ to the Cauchy problem
\begin{equation}
\label{CP}
\left\{ \begin{array}{rcll}
\displaystyle  D_tu(t,x) &\hs{5} =\hs{5} & \A (t)u(t,x), & t>s,\; x \in   \R^{d},\\[2mm]
u(s,x) &\hs{5} =\hs{5} & \varphi(x), & x\in \R^d.
\end{array}\right.
\end{equation}
Since the coefficients are allowed to be unbounded, $L^p$ spaces
with respect to the Lebesgue measure are not a natural setting for
problem \eqref{CP}. This is well understood in the autonomous case
$\A(t) \equiv A$, where $P(t,s) = e^{(t-s)A}$ and the Lebesgue
measure is replaced by an invariant measure, i.e., a Borel
probability measure $\mu$ such that
\begin{eqnarray*}
\int_{\R^d} e^{tA}\varphi\, d\mu = \int_{\R^d} \varphi\, d\mu, \quad t>0, \; \varphi \in C_b(\R^d).
\end{eqnarray*}
Under suitable assumptions it is possible to show that there exists a
unique  invariant measure. In this case, $e^{tA}$ is extended to a
contraction semigroup in $L^p(\R^d, \mu)$ for every $p\in [1,
\infty)$, and $e^{tA}\varphi $ goes to the mean value $\int_{\R^d}
\varphi \, d \mu$ in $L^p(\R^d, \mu)$ for every $\varphi \in
L^p(\R^d, \mu)$ as $t\to \infty$, if $p>1$.

The natural generalization of invariant measures to the time
depending case are fa\-milies of Borel probability measures
$\{\mu_s:\, s\in \R \}$, called {\em evolution systems of measures}, such
that
\begin{eqnarray*}
\int_{\R^d} P(t,s)\varphi\, d\mu_t = \int_{\R^d} \varphi\, d\mu_s,
\quad t>s, \; \varphi \in C_b(\R^d).
\end{eqnarray*}
A sufficient condition for their existence, similar to a well known sufficient condition for the existence
of an invariant measure in the autonomous case, is the following: there exist a
$C^2$ function $V:\R^d\to \R$ such that $\lim_{|x|\to \infty}V(x) = +\infty$,
and positive numbers $a$, $c$ such that $\A(s)V(x) \leq a - cV(x)$ for each $s\in \R$ and $x\in \R^d$.

If an evolution system of measures exists, then,
as in the autonomous case, $P(t,s)$ may be extended to a contraction  (still called $P(t,s)$)
from $L^p(\R^d, \mu_s)$ to $L^p(\R^d, \mu_t)$, i.e.,
\begin{equation}
\label{contraction}
 \| P(t,s)\varphi  \|_{L^p(\R^d, \mu_t)} \leq  \|  \varphi  \|_{L^p(\R^d, \mu_s)}, \quad t>s,
\end{equation}
for every $\varphi\in L^p(\R^d, \mu_s)$.

In this paper we treat the case of time periodic coefficients, and
we study asymptotic behavior of $P(t,s)$ and spectral properties of
the parabolic operator
\begin{equation}
\label{G}
\G : = \A(t) - D_t
\end{equation}
in $L^p$ spaces associated to a distinguished  evolution system of
measures. In fact, the evolution systems of measures are infinitely
many, and we consider the unique $T$-periodic one, i.e. the only one
such that $\mu_s= \mu_{s+T}$ for every $s\in \R$,  where $T$ is the period of the
coefficients.
We extend to this setting  the convergence results of
the autonomous case, showing that for  $1<p<\infty$
\begin{equation}
\label{CompAs} \lim_{t\to \infty} \| P(t,s)\varphi -
m_s\varphi\|_{L^p(\R^d, \mu_t)} =0, \quad s\in \R, \; \varphi \in
L^p(\R^d, \mu_s),
\end{equation}
and
\begin{equation}
\label{CompAsInd} \lim_{s\to -\infty} \| P(t,s)\varphi -
m_s\varphi\|_{L^p(\R^d, \mu_t)} =0, \quad t\in \R, \; \varphi \in
C_b(\R^d),
\end{equation}
under suitable assumptions, that in the case of $C^1_b$ diffusion coefficients reduce to
\begin{equation}
\label{DissB}
\sup_{s\in \R, \;x, y\in \R^d, \,x\neq y} \frac{\langle b(s,x) - b(s,y), x-y\rangle }{|x-y|^2} < \infty ,
\end{equation}
or, equivalently, $\sup \{\sum_{i,j=1}^{d} D_ib_j(s,x)\xi_i\xi_j:$
$(s,x)\in \R^{1+d}$, $\xi\in \R^d,$ $|\xi| =1\} <\infty$. This can
be seen as a weak dissipativity condition on the vector fields
$b(s,\cdot)$. Under a stronger dissipativity condition, for bounded
diffusion coefficients  we prove exponential convergence, i.e., for
every $p\in (1, \infty)$ there exist $M >0$, $\omega <0$ such that
\begin{equation}
\label{CompAsExp} \| P(t,s)\varphi - m_s\varphi\|_{L^p(\R^d, \mu_t)}
\leq Me^{\omega (t-s)} \|\varphi\|_{L^p(\R^d, \mu_s)} , \quad t>s,
\; \varphi \in L^p(\R^d, \mu_s).
\end{equation}
The stronger dissipativity assumption was used in \cite{KLL} to
prove pointwise gradient estimates for $P(t,s)\varphi$. In fact, we
arrive at exponential convergence through gradient estimates. Then,
we discuss the rate of convergence;  in Theorem \ref{gamma=omega} we
show  that for $p\geq 2$ and $\omega \in \R$ the conditions
\begin{enumerate}[\rm (a)]
\item
$\exists M\hskip -1pt>\hskip -1pt0 :$ $\| P(t,s)\varphi -
m_s\varphi\|_{L^p(\R^d, \mu_t)}\hskip -1pt \leq\hskip -1pt M
e^{\omega (t-s)} \|\varphi\|_{L^p(\R^d, \mu_s)}$, $  t>s$,
$\varphi\in L^p(\R^d, \mu_s)$,\vspace{2mm}
\item
$\exists N\hskip -1pt>\hskip -1pt0 :$ $ \| \,|\hskip -.5pt\nabla_x
P(t,s)\varphi |\hskip -.5pt\, \|_{L^p(\R^d, \mu_t)}\hskip -1pt
\leq\hskip -1pt N e^{\omega (t-s)} \|\varphi\|_{L^p(\R^d, \mu_s)}$,
$  t>s+1$, $\varphi\in L^p(\R^d, \mu_s)$,
\end{enumerate}
are equivalent. Therefore, denoting by $\omega_p$ (resp. $\gamma_p$) the infimum of the $\omega \in \R$ such that (a) (resp. (b)) holds,
we have $\omega_p = \gamma_p$.

Such characterization of the convergence rate was proved for time
depending Ornstein-Uhlenbeck operators (i.e., when $Q$ is
independent of $x$ and $B$ is linear in $x$)   in \cite{GL2} for
$p=2$. Apart from Ornstein-Uhlenbeck operators, it seems to be new
even in the autonomous case.
For Ornstein-Uhlenbeck operators we have a precise expression of $\gamma_2$ in terms of the data, and our Theorem  \ref{Th:cons}  shows that  $\gamma_p =\gamma_2 <0$ for every $p\in (1, \infty)$. In general, $\gamma_p$
could depend explicitly on $p$  and we only give upper estimates for it.

In the autonomous case, exponential convergence to equilibrium in
$L^2(\R^d, \mu)$ is usually obtained through Poincar\'e inequalities
such as
\begin{equation}
\label{PoincAuton}
\int_{\R^d} \left |\varphi - \int_{\R^d} \varphi\,d\mu \right |^2 d\mu \leq C_0 \int_{\R^d} |Q^{1/2}\nabla \varphi|^2 d\mu, \quad \varphi \in D(A),
\end{equation}
where $D(A)$ is the domain of the generator of   $e^{tA}$ in
$L^2(\R^d,\mu)$. If  \eqref{PoincAuton} holds we get $\omega_2 \leq
-\eta_0/C_0$ and in the (symmetric) case $A\varphi = \Delta \varphi
+ \langle \nabla \Phi, \nabla \varphi\rangle$ we have $\omega_2 =
1/C_0 = \eta_0/C_0$, and $\omega_2$ is a minimum. Therefore, the
problem is reduced to find the best Poincar\'e constant $C_0$, which
is a hard task in general. The upper bounds on $C_0$ that come from
gradient estimates yield $\eta_0/C_0 \geq \gamma_2$, and the
equality holds only in very special cases. Therefore, Theorem
\ref{gamma=omega} gives a better rate of convergence  (see the
discussion after Corollary \ref{Cor:ca}).

We follow a purely deterministic approach, although the well known connections between linear second order parabolic equations and nonlinear ordinary stochastic differential equations might be used (such as e.g. in \cite{SV,EK,C}) to get some of our formulae and/or  estimates.
The key tool of our analysis is  the evolution semigroup,
\begin{eqnarray*}
\T(t) u (s,x) = P(s, s-t)u(s-t, \cdot)(x), \quad t\geq 0, \; s\in
\R, \;x\in \R^d,
\end{eqnarray*}
that is a Markov semigroup in the space $C_b(\U \times \R^d)$ of the
continuous and bounded functions $u$ such that $u(s,\cdot) = u(s+T,
\cdot)$ for all $s\in \R$. Its unique invariant measure is
\begin{eqnarray*}
\mu(ds, dx) = \frac{1}{T} \, \mu_s(dx)ds.
\end{eqnarray*}
All Markov semigroups having invariant measures have natural
extensions to contraction semigroups in $L^p$ spaces with respect to
such measures. Dealing with periodic functions, we consider the
space $L^p(\U \times \R^d, \mu)$ that consists of all
$\mu$-measurable functions $u$ such that $u(s, \cdot) = u(s+T,
\cdot)$ for a.e. $s\in \R$, and such that $\int_{0}^{T}\int_{\R^d}
|u(s,x)|^p \mu_s(dx)\,ds$ is finite. We denote by $G_p$ the
infinitesimal generator of $\T(t)$ in $L^p(\U \times \R^d, \mu)$.
$G_p$ is a realization of the parabolic operator $\G$, defined in
\eqref{G}, in the space $L^p(\U \times \R^d, \mu)$.

We introduce a projection  $\Pi$  on space independent functions,
\begin{eqnarray*}
\Pi u(s,x): = \int_{\R^d} u(s,y)\mu_s(dy), \quad s\in \R,\;\,x\in\R^d,
\end{eqnarray*}
and we prove that $\T(t)(I-\Pi)$ is strongly stable in all spaces $L^p(\U \times \R^d, \mu)$, $1<p<\infty$, that is
\begin{equation}
\label{CompAsintT(t)}
 \lim_{t\to \infty }\| \T(t) (u-\Pi u) \|_{L^p(\U \times \R^d, \mu)} =0, \quad u\in L^p(\U \times \R^d, \mu).
\end{equation}
From this fact we deduce \eqref{CompAs} and \eqref{CompAsInd}; if
$\T(t)(I-\Pi)$ is exponentially stable we deduce \eqref{CompAsExp}.
We arrive at \eqref{CompAsintT(t)} through a similar property of the
space gradient of $\T(t) u$, i.e.,
\begin{eqnarray*}
\lim_{t\to \infty }\| \,| \nabla_x\T(t) u|\, \|_{L^p(\U \times \R^d, \mu)} = 0, \quad u\in L^p(\U \times \R^d, \mu),
\end{eqnarray*}
which is proved using semigroups arguments that seem not  to have
counterparts for evolution operators. In particular, we use the
identity
\begin{eqnarray*}
\int_{0}^{T}\int_{\R^d } u\, G_2 u \, d\mu = - \int_{0}^{T}\int_{\R^d} \langle Q\nabla_x u, \nabla_xu \rangle d\mu ,\quad u\in D(G_2),
\end{eqnarray*}
which is a time dependent version of what is called  {\em identit\'e de carr\`e du champ} by the french mathematicians.

$\T(t)$ is a nice example of a Markov semigroup that is not strong Feller and not irreducible, and that  has a unique invariant measure $\mu$. On the other hand, \eqref{CompAsintT(t)} shows that in general $\T(t)u$ does not converge to the mean value of $u$ with respect to $\mu$ as $t\to \infty$.

Together with asymptotic behavior results, it is natural to get spectral properties of the operators $G_p$, $1<p<\infty$. When  \eqref{CompAsExp} holds, we prove that $G_p$ has a spectral gap, and precisely
\begin{eqnarray*}
\sup \,\{ \mbox{\rm Re}\; \lambda: \lambda \in \sigma (G_p)\setminus i\R \} = \omega_p <0.
\end{eqnarray*}

We remark that the equation  $\lambda u -\G u =f$, with $f\in L^p(\U \times \R^d, \mu)$,  cannot be seen as an evolution equation in a fixed $L^p$ space $X$,
\begin{eqnarray*}
u'(t) - \A(t) u (t) + \lambda u(t) = f(t, \cdot)
\end{eqnarray*}
because our spaces $X(t) = L^p(\R^d, \mu_t)$ vary with time. For the same reason, $\T(t)$ is not a usual evolution semigroup in a fixed Banach space $X$. However, it exhibits some of the typical features of evolution semigroups in fixed Banach spaces, in particular the Spectral Mapping Theorem holds.

If the diffusion coefficients do not depend on the space variables, and the supremum in \eqref{DissB} is equal to some negative number $r_0$,   we get a log-Sobolev type inequality,
\begin{equation}
\label{LogSob2} \int_{0}^{T}  \int_{   \R^d } |u|^2 \log (u^2) d \mu
\leq \frac{1}{T}\int_0^T \Pi u^2 \log (\Pi u^2)ds + \frac{2
\Lambda}{  |r_0|} \int_{0}^{T}  \int_{   \R^d }  |\nabla_x u|^2
d\mu,
\end{equation}
for every $u\in D(G_2)$. Here, $\Lambda$ is the supremum of the
maximum eigenvalues of the matrices $Q(s)$ when $s$ varies in
$[0,T]$. Also this inequality is proved using the evolution
semigroup $\T(t)$,  through semigroups arguments that have no
counterparts for evolution operators. Using \eqref{LogSob2} we show
that the domains $D(G_p)$ are compactly embedded in $L^p(\U \times
\R^d, \mu)$ for $1<p<\infty$, and from this fact a lot of nice
consequences follow. In particular, the spectrum of each operator
$G_p$ consists of eigenvalues and it is independent of $p$, and  the
exponential decay rates $\omega_p \leq r_0$ are independent of $p$.

The interest in log-Sobolev estimates goes beyond asymptotic
behavior, and much literature has been devoted to them in the
autonomous case. See e.g., the surveys \cite{many,Gross}.  Therefore,
it is worth to establish them in $L^p$ spaces with time-space
variables. In Proposition \ref{dsharp any} and in Theorem
\ref{LogSob_dominio} we prove $L^p$ versions of \eqref{LogSob2}.

The paper ends with illustrations of the asymptotic behavior and
spectral results for explicit examples of families of operators
$\A(t)$ that satisfy our assumptions.

Except for nonautonomous Ornstein-Uhlenbeck operators, this one
seems to be the first systematic study of asymptotic behavior in
linear nonautonomous parabolic problems with unbounded coefficients
in $\R^d$. A part of our results lends itself to generalizations  to
some infinite dimensional settings, where $\R^d$ is replaced by a
separable Hilbert space $H$, in the spirit of e.g.,
\cite{DPZ,DPZbrutto,DPR}.

\subsection*{Notations}
We denote by $B_b(\R^d)$ the Banach space of all bounded and
Borel  measurable functions
$f:\R^d\to\R$, and by $C_b(\R^d)$ its subspace of all  continuous
functions.   $B_b(\R^d)$ and $C_b(\R^d)$ are endowed with the
sup norm $\|\cdot\|_{\infty}$. For $k\in \N$, $C^k_b(\R^d)$ is the
set of all functions $f\in C_b(\R^d)$ whose derivatives up to the
$k$th-order are bounded and continuous in $\R^d$. We use the
subscript ``$c$'' instead of ``$b$''  for spaces of functions  with compact support.


Throughout the paper we  consider real valued functions
$(s,x)\mapsto f(s,x)$ defined in $\R^{1+d}$, that are $T$-periodic
with respect to time. It is useful to identify such functions with
functions defined in $\U \times \R^d$, where $\U = [0,T]$    {\em
mod} $T$. So, we denote by $C_b(\U \times \R^d)$ the space of the
continuous, bounded, and $T$-time periodic functions $f:\R^{1+d}\to
\R$, endowed with the sup norm.
Similarly, for any $\alpha\in (0,1)$, we denote by
$C^{\alpha/2,\alpha}_{\rm loc}(\U\times\R^d)$ the set of all
functions $f\in C_b(\U\times\R^d)$ which belong to
$C^{\alpha/2,\alpha}([0,T]\times B(0,R))$ for any
$R>0$, and by $W^{1,2}_{p,\rm loc}(\U \times \R^d,ds\times dx)$ the set of all time periodic functions $f$ such that
$f$, $D_sf$, and the first and second order space derivatives of $f$ belong to $L^{p}((0,T)\times B(0,R), ds\times dx)$ for any
$R>0$.


 \section{Preliminaries}
 \label{preliminaries}
\setcounter{equation}{0}

\subsection{General properties of $P(t,s)$ and of evolution systems of measures}
\label{sub1}

\begin{hyp}\label{hyp1}
\begin{enumerate}[\rm (i)]
\item
The coefficients $q_{ij}$ and $b_i$ $(i,j=1,\ldots,d)$ are $T$-time periodic and belong to
$C^{\alpha/2,\alpha}_{\loc}(\U \times
 \R^{d})$ for any $i,j=1,\ldots,d$
and some $\alpha\in (0,1)$.
\item
For every $(s,x)\in  \R^{1+d}$,
the matrix $Q(s,x)$ is symmetric and there exists a function
$\eta: \U \times\R^d\to \R$ such that
$0<\eta_0:=\inf_{\U \times\R^d}\eta$ and
\begin{eqnarray*}
\langle Q(s,x)\xi , \xi \rangle \geq \eta(s,x) |\xi |^2, \qquad\;\,
\xi \in \R^d,\;\,(s,x) \in \U \times \R^{d}.
\end{eqnarray*}
\item
There exist a positive function $V  \in C^2(\R^d)$ and numbers
$a$,  $c>0$    such that
\begin{eqnarray*}
\qquad\;\, \lim_{|x|\to \infty} V (x) = \infty \quad \mbox{
and } \quad (\A(s) V)( x) \leq a - c V(x), \quad (s,x)\in  \U \times \R^{d}.
\end{eqnarray*}
\end{enumerate}
\end{hyp}

Here we recall some results from \cite{KLL} and \cite{LZ}. The first
one is that, under Hypothesis \ref{hyp1}(i)(ii), for every $f\in
C_b(\R^d)$ problem \eqref{CP} has a unique bounded classical
solution $u$. The evolution operator $P(t,s)$ is defined by
\begin{eqnarray*}
P(t,s)f  = u(t, \cdot), \quad t\geq s \in \R.
\end{eqnarray*}
Some properties of $P(t,s)$, taken from \cite{KLL}, are summarized
in the next theorem  and in its corollaries.

\begin{theorem}
\label{Th:P(t,s)} Let Hypothesis  \ref{hyp1} hold. Define $ \Lambda
:= \{(t,s,x)\in\R^{2+d}: t > s, \;x\in \R^{d}\} $. Then:
\begin{enumerate}[{\rm (i)}]
\item
for every $\varphi \in C_b(\R^d)$, the function $(t,s,x)\mapsto
P(t,s)\varphi(x)$ is continuous in $\overline{\Lambda }$. For every
$s\in \R$, the function $(t,x)\mapsto P(t,s)\varphi(x)$ belongs to
$C^{1+\alpha/2, 2+\alpha}_{\loc}((s, \infty)\times \R^d )$;
\item
for every $\varphi\in C^{\infty}_{c}( \R^d )$, the function
$(t,s,x)\mapsto P(t,s)\varphi(x)$ is continuously differentiable
with respect to $s$ in $\overline{\Lambda }$ and
$D_sP(t,s)\varphi(x) =   -P(t,s)\A(s)\varphi(x)$ for any
$(t,s,x)\in\overline\Lambda$;
\item
for each $(t,s,x)\in \Lambda$ there exists a Borel probability measure $p_{t,s,x}$ in $\R^d$ such that
\begin{equation}
\label{representation}
P(t,s)\varphi(x) = \int_{\R^d}\varphi (y)
p_{t,s,x}(dy), \quad f\in C_b(\R^d).
\end{equation}
Moreover, $ p_{t,s,x}(dy) = g(t,s,x,y)dy$ for a positive function
$g$. In particular, $P(t,s)$ is irreducible;
\item
$P(t,s)$ is strong Feller;  extending it to  $L^{\infty}(\R^d, dx)$ through  formula \eqref{representation}, it maps
  $L^{\infty}(\R^d, dx)$ (and, in particular, $B_b(\R^d)$) into $C_b(\R^d)$ for $t>s$, and
\begin{eqnarray*}
\|P(t,s)\varphi \|_{\infty} \leq \|\varphi \|_{\infty}, \quad
\varphi \in L^{\infty}(\R^d, dx), \;\,t>s;
\end{eqnarray*}
\item
there exists a tight$^($\footnote{i.e.,   $\forall \eps
>0$   $\exists R=R(\eps)>0$ such that
$
\mu_s(B(0,R)) \geq 1-\eps$, for all $s\in\R$.}$^)$
 evolution system of measures $\{ \mu_s: s\in
\R\}$ for $P(t,s)$.  Moreover,
\begin{equation}
\label{P(t,s)V}
P(t,s)V(x) := \int_{\R^d}V(y) p_{t,s,x}(dy) \leq
V(x) + \frac{a}{c}, \quad t>s, \;x\in \R^d,
\end{equation}
and
\begin{equation}
\label{m_tV}
  \int_{\R^d}V(y)\mu_t(dy) \leq \min V +
\frac{a}{c}, \quad t\in \R ,
\end{equation}
where the constants $a$ and $c$ are given by Hypothesis
\ref{hyp1}(iii). \end{enumerate}
\end{theorem}

Note that estimate \eqref{P(t,s)V} implies that the family $\{
p_{t,s,x}: t>s, \;x\in B(0,r)\}$ is tight for every $r>0$.

Since the coefficients $q_{ij}$ and $b_i$ are $T$-time periodic,
uniqueness of the bounded solution to  \eqref{CP} implies that
$P(t+T,s+T) = P(t,s)$ for $t\geq s$. Moreover, looking at the
construction of  the measures $\mu_t$ of \cite{KLL} one can see that
$\mu_s = \mu_{s+T}$ for each $s\in \R$ (\cite[Rem. 6.8(i)]{LZ}).

The evolution systems of invariant measures are infinitely many, in
general. In the case of nonautonomous Ornstein-Uhlenbeck equations,
they  have been explicitly characterized in \cite[Prop. 2.2]{GL1}.
In the next section we shall prove that all the $T$-periodic
families $\{\mu_s: s\in \R\}$ constructed in \cite{KLL} actually
coincide, since $P(t,s)$ has a unique $T$-periodic evolution system
of measures.

In the next corollary we prove some consequences of   Theorem
\ref{Th:P(t,s)}. For this purpose, for every $\varphi \in
L^1(\R^d,\mu_s)$ we define the mean value
\begin{eqnarray*}
m_s\varphi:= \int_{\R^d}\varphi (y)\mu_s(dy).
\end{eqnarray*}

\begin{corollary}
\label{cor:continuity} Let  Hypothesis  \ref{hyp1} hold. Then:
\begin{itemize}
\item[(a)]
for every
$\varphi \in C_b(\R^d)$ the function $s\mapsto m_s\varphi  $ is
continuous in $\R$. More generally, for every $u\in C_b(\R^{1+d})$
the function $s\mapsto m_su(s, \cdot)$ is continuous in
$\R$;
\item[(b)]
for every $\varphi \in C_b(\R^d)$ and for every bounded sequence $(\varphi_n)\subset C_b(\R^d)$ that converges locally uniformly to
$\varphi$ we have
\begin{equation}
\label{lim_unif}
\lim_{n\to \infty} \sup_{s\in \R}  \|\varphi-\varphi_n\|_{L^p(\R^d,\mu_s)}  =0, \quad 1\leq p<\infty ,
\end{equation}
 and, for every $r>0$,
\begin{equation}
\label{lim_unif_P(t,s)} \lim_{n\to \infty} \sup_{s\leq t \in \R}
\|P(t,s) (\varphi-\varphi_n)\|_{L^{\infty}(B(0,r))}  =0.
\end{equation}
\item[(c)] For $t>s$, $P(t,s)$ may be extended to a bounded   operator from $L^p (\R^d, \mu_s)$ to $L^p (\R^d, \mu_t)$ for all  $p\in [1, \infty)$, and \eqref{contraction} holds.
\end{itemize}
\end{corollary}
\begin{proof}
The first part of   statement (a) is an easy consequence of the
continuity of $P(t,s)\varphi $ with respect to $s$. Indeed, fix
$s_0\in \R$ and $t\geq s_0+1$. For $s\in (s_0-1, s_0+1)$ we have
\begin{eqnarray*}
m_s\varphi  - m_{s_0}\varphi   =  \int_{\R^d}(P(t,s)\varphi  (y)-
P(t,s_0)\varphi (y))\mu_t(dy).
\end{eqnarray*}
By Theorem \ref{Th:P(t,s)}(i), for every $y\in \R^d$ we have
$\lim_{s\to s_0}P(t,s)\varphi  (y)- P(t,s_0)\varphi (y)=0$; moreover
$|P(t,s)\varphi  (y)- P(t,s_0)\varphi(y)|\leq 2\|f\|_{\infty}$.
Therefore, $\lim_{s\to s_0} m_s\varphi  - m_{s_0}\varphi   =0$.

Let us prove the second part of   statement (a). Fix $s_0\in \R$.
Then,
\begin{eqnarray*}
|m_su(s, \cdot) - m_{s_0}u(s_0, \cdot) | \leq  |
m_s (u(s, \cdot) - u(s_0, \cdot))| + |m_s u(s_0, \cdot) - m_{s_0}u(s_0, \cdot) |, \quad s\in \R.
\end{eqnarray*}
By the first part of the statement, $\lim_{s\to s_0}  |m_s u(s_0,
\cdot) - m_{s_0}u(s_0, \cdot) | =0$.  To estimate $| m_s (u(s,
\cdot) - u(s_0, \cdot))|$ we use  Theorem  \ref{Th:P(t,s)}(v). Given
$\eps
>0$, let $R>0$ be such that $\mu_s(\R^d \setminus B(0,R)) \leq \eps$
for every $s\in \R$. Then,
\begin{align*}
&  | m_s (u(s, \cdot) - u(s_0, \cdot))|\\[1mm]
&  \leq  \int_{B(0,R)} |u(s,y) - u(s_0,y)|\mu_s(dy) +  \int_{\R^d \setminus B(0,R)} |u(s,y) - u(s_0,y)|\mu_s(dy)
\\[1mm]
& \leq   \|u(s, \cdot) - u(s_0, \cdot)\|_{L^{\infty} (B(0,R))} +
2\eps\|u\|_{\infty}.
\end{align*}
Since $u$ is continuous, $ \|u(s, \cdot) - u(s_0,
\cdot)\|_{L^{\infty}(B(0,R))} \leq \eps$ for $|s-s_0|$ small enough.
Statement (a) follows.

The proof of statement (b) is similar.
Let  $M>0$ be such that  $\| \varphi_n\|_{\infty} \leq M$ for each  $n\in \N$. For every $\eps >0$ let $R>0$ be as above. Then,
\begin{align*}
\int_{\R^d}|\varphi-\varphi_n|^pd\mu_s
&=
\int_{B(0,R)}|\varphi-\varphi_n|^pd\mu_s +\int_{\R^d\setminus
B(0,R)}|\varphi-\varphi_n|^pd\mu_s
\\
& \le \|\varphi-\varphi_n\|_{L^{\infty}( B(0,R))} +
 (\|\varphi\|_{\infty} + M)^p\varepsilon .
\end{align*}
and \eqref{lim_unif} holds. The proof of \eqref{lim_unif_P(t,s)} is
the same, through the representation formula $P(t,s)\varphi(x) =
\int_{\R^d}\varphi(y)  p_{t,s,x}(dy)$ and the tightness of $\{
p_{t,s,x}: s<t, \;x\in B(0,r)\}$.

The proof of statement (c) is the same of the autonomous case. Indeed, for every $\varphi \in C_b(\R^d)$ we have, by Theorem \ref{Th:P(t,s)}(iii) and the H\"older inequality,
\begin{eqnarray*}
|P(t,s)\varphi (x)|^p  = \bigg| \int_{\R^d} \varphi(y)
p_{t,s,x}(dy)\bigg|^p \leq\int_{\R^d}|\varphi(y)|^p p_{t,s,x}(dy) =
P(t,s)|\varphi |^p(x),
\end{eqnarray*}
so that, integrating with respect to $\mu_t$, we get
\begin{eqnarray*}
\int_{\R^d} |P(t,s)\varphi |^p d\mu_t \leq \int_{\R^d}
P(t,s)|\varphi |^p d\mu_t = \int_{\R^d}  |\varphi |^p d\mu_s , \quad
t\geq s, \end{eqnarray*} i.e., $\varphi$ satisfies
\eqref{contraction}. Since $C_b(\R^d)$   is dense in $L^p(\R^d,
\mu_s)$, \eqref{contraction} holds for every $\varphi \in L^p(\R^d,
\mu_s)$.
\end{proof}

\subsection{Smoothing properties of   $P(t,s)$}

We recall some global smoothing properties of the evolution operator
$P(t,s)$ that have been proved in \cite{KLL,LZ} and will be
extensively used in this paper.

\begin{hyp}\label{hyp2}
\begin{enumerate}[{\rm (i)}]
\item
The first-order space derivatives of the data $q_{ij}$ and $b_i$ $(i,j=1,\ldots,d)$ exist and belong to $C^{\alpha /2, \alpha}_{\rm
loc}(\U \times \R^d )$;
\item
there are two upperly bounded functions $\zeta : \U \to \R_+$ and
$r: \U \times \R^{d}\to\R$ such that
\begin{eqnarray*}
\begin{array}{lll}
\qquad\;\;\;\;\;\;\;\langle \nabla_x b(s,x)\xi,\xi \rangle \leq r(s,x)|\xi|^2,
\quad & (s,x) \in \U \times \R^d , & \xi \in \R^d,\\[1.5mm]
\qquad\;\;\;\;\;\;\;|D_kq_{ij}(s,x)| \leq \zeta(s)\eta(s,x), \quad & (s,x) \in \U \times \R^d , &
  i,j, k=1,\dots,d,
\end{array}
\end{eqnarray*}
where $\eta(s,x)$ is the ellipticity constant at $(s,x)$ in Hypothesis
$\ref{hyp1}(ii)$.
\end{enumerate}
\end{hyp}

The following theorem has been proved in \cite{KLL}.

\begin{theorem}
\label{Th:gradP(t,s)infinito} Let Hypotheses  \ref{hyp1} and
\ref{hyp2} hold. Then, there exist positive constants $C_1,C_2>0$,
such that
\begin{enumerate}[{\rm (i)}]
\item
for every $\varphi \in C^1_b(\R^d)$ we have
\begin{eqnarray*}
\| \nabla_x P(t,s) \varphi  \|_{\infty} \leq C_1\|\varphi
\|_{C^1_b(\R^d)}, \qquad\;\, s<t\leq s+1;
\end{eqnarray*}
\item
for every $\varphi  \in C_b(\R^d)$ we have
\begin{equation}
\|\nabla_x P(t,s)\varphi  \|_{\infty} \leq \frac{C_2}{\sqrt{t-s}}\|
\varphi \|_{\infty}, \qquad\;\, s< t \leq s+1. \label{star-luca}
\end{equation}
\end{enumerate}
\end{theorem}

As a consequence, we obtain
 \begin{equation}
 \label{boh}
\|\nabla_x P(t,s)\varphi \|_{\infty} \leq C_2 \|\varphi \|_{\infty},
\qquad\;\, t\geq s+1,
\end{equation}
for every $\varphi  \in C_b(\R^d)$. It is sufficient to recall that
$P(t,s)\varphi  = P(t,t-1)P(t-1,s)\varphi $ and that
$\|P(t-1,s)\varphi \|_{\infty} \leq \|\varphi \|_{\infty}$.

\begin{theorem}
\label{Th:gradP(t,s)puntuale} Let Hypotheses  \ref{hyp1} and
\ref{hyp2} hold and assume in addition that, for some $p>1$,
\begin{equation}
\ell_p :=  \sup_{(s,x)\in [0,T]\times\R^d}\left(r(s,x) +
\frac{d^3(\zeta(s))^2\eta(s,x)}{4\min\{p -1,1\}}\right)<\infty.
\label{ell_p}
\end{equation}
Then:
\begin{enumerate}[{\rm (i)}]
\item
for every $\varphi \in C^1_b(\R^d)$  we have
\begin{equation}
|\nabla_x P(t,s)\varphi(x)|^p \leq e^{p\ell_p(t-s)}P(t,s)|\nabla
\varphi |^p(x), \qquad\;\, t\geq s,\;\,x\in\R^d; \label{grad-punt}
\end{equation}
\item
there exists a positive constant $C_3= C_3(p)$ such that
\begin{equation}
|\nabla_x P(t,s)\varphi(x)|^p \leq C_3^p \max\{(t-s)^{-p/2}, \,1\}
e^{p\ell_p(t-s)}P(t,s)|\varphi |^p(x), \label{grad-punt-1}
\end{equation}
 for every $\varphi \in
C_b(\R^d)$,    $t> s$ and   $x \in \R^d$;
\item
if the diffusion coefficients $q_{ij}$ $(i,j=1,\ldots,d)$ are
independent of $x$, then
\eqref{grad-punt} holds for $p=1$ too, with
$\ell_1=r_0:=\sup_{(s,x)\in \U \times\R^d}r(s,x)$. Moreover, there
exists $C_4>0$, independent of $t$ and $s$, such that
\begin{equation}
\|\,|\nabla_x P(t,s)\varphi |\,\|_{\infty} \leq
C_4e^{r_0(t-s)}\|\varphi \|_{\infty}, \qquad\;\, t\geq s+1.
\label{grad-unif}
\end{equation}
\end{enumerate}
\end{theorem}
\begin{proof}
Estimates \eqref{grad-punt} and \eqref{grad-unif} have been proved
in \cite{KLL}. To be precise, in \cite[Cor. 4.6]{KLL} estimate
\eqref{grad-unif} is stated as $\|\,|\nabla_x P(t,s)\varphi
|\,\|_{\infty} \leq  Ce^{\ell_p (t-s)}\|\varphi \|_{\infty}$, with
$C$ independent of $p$. If the diffusion coefficients are
independent of $x$, we can take $\zeta\equiv 0$. Hence, $\ell_p=r_0$
and \eqref{grad-unif} follows.

In (the proof of) \cite[Prop.~3.3]{LZ}, an estimate similar to
\eqref{grad-punt-1} has been proved with a worse exponential term.
To get \eqref{grad-punt-1} it is sufficient to observe that  for
$t-s\leq 1$,  \cite[Prop.~3.3]{LZ} gives

\begin{equation}
\label{1} |\nabla_x P(t,s)\varphi(x)|^p \leq
\frac{K^p}{(t-s)^{\frac{p}{2}}}P(t,s)|\varphi |^p(x), \qquad\;\, x
\in \R^d,
\end{equation}
for some positive constant $K=K(p)$, independent of $s,t$ and $\varphi$. If
$t-s>1$, we write $P(t,s)\varphi =P(t,s+1)P(s+1,s)\varphi $. From
\eqref{grad-punt} and \eqref{1}  we   obtain
\begin{align*}
|\nabla_x P(t,s)\varphi(x)|^p&\le e^{p\ell_p(t-s-1)}P(t,s+1)|\nabla_xP(s+1,s)\varphi |^p(x)\\
&\le K^pe^{p\ell_p(t-s-1)}P(t,s+1)P(s+1,s)|\varphi |^p(x)\\
&= K^pe^{p\ell_p(t-s-1)}P(t,s)|\varphi |^p(x),
\end{align*}
for any $x\in\R^d$. Estimate \eqref{grad-punt-1} follows.
\end{proof}

\begin{remark}
{\rm Two remarks are in order.
\begin{enumerate}[(a)]
\item
Estimate \eqref{grad-punt-1} implies that $ P(t,s)$ maps
$L^p(\R^{d}, \mu_s)$ into $W^{1,p}(\R^{d}, \mu_t)$ for $p>1$, and
\begin{equation}
\label{grad-p} \|\,|\nabla_x P(t,s)\varphi |\, \|_{L^p(\R^d, \mu_t)}
\leq C_3 \max\{ (t-s)^{-1/2}, \,1\} e^{ \ell_p(t-s)} \|  \varphi
\|_{L^p(\R^d, \mu_s)}.
\end{equation}
This is not true in general for $p=1$, even in the autonomous case. See e.g., \cite{MPP}
for a counterexample given by the  Ornstein-Uhlenbeck semigroup.
\item
Estimates   \eqref{grad-p}   are sharp near $t=s$, but they are not
for $t\gg s$  if $\ell_p>0$. In this case for   $t>s+1$ we write
$P(t,s)  = P(t,t-1)P(t-1,s)$, and using \eqref{contraction}   we
obtain
\begin{equation}
\label{grad-p>1}
\;\;\;\|\,|\nabla_x P(t,s)\varphi|\,\|_{L^p(\R^{d},
\mu_t)} \leq C_3e^{\ell_p}\|\varphi \|_{L^p(\R^{d}, \mu_s)} , \quad
t\geq s+1, \; \varphi \in L^p(\R^{d}, \mu_s).
\end{equation}
\end{enumerate}
}
\end{remark}


\subsection{The evolution semigroup}

The evolution semigroup $\T(t)$ is defined on continuous and bounded functions $f$ by
\begin{eqnarray*}
\T(t)f(s,x)=P(s,s-t)f(s-t,\cdot)(x),\qquad\;\,(s,x)\in  \R^{1+d}
,\;\,t\ge 0.
\end{eqnarray*}
In \cite[Prop.~6.1]{KLL} we have shown that $\T(t)$ is a semigroup
of positive contractions   in $C_b( \R^{1+d} )$. Since $
P(s+T,s+T-t) = P(s,s-t)$, $ \T(t)$ leaves $C_{b} (\U \times  \R^{d
})$ invariant for every $t>0$.

$\T(t)$ is not strongly continuous in $C_{b} (\U \times \R^d
)$. However,  the last part of
the proof of \cite[Prop.~6.1]{KLL} implies that, for any $f\in
C_{b} (\U \times \R^d )$ and any $t_0\ge 0$, $\T(t)f$ tends to
$\T(t_0)f$, locally uniformly in $\U \times \R^d $ as $t\to t_0$.

In the language of \cite{DPZ}, $\T(t)$ is a stochastically
continuous Markov semigroup. It improves spatial regularity, as the
next lemma shows.

\begin{lemma}
\label{le:derivateT(t)}
For every $t>0$ and $f\in C_{b} (\U \times \R^d )$, the  derivatives $D_{i}\T(t)f$, $D_{ij}\T(t)f$ exists and are continuous in $\U\times\R^d$ for $i, j=1, \ldots ,d$.
\end{lemma}
\begin{proof} Since
$\T(t)f(s,x) = P(s,s-t)f(s-t, \cdot)(x)$, it is sufficient to show
that the first and second order space derivatives of the function
$(t,r,x)\mapsto P(t,r)f(r, \cdot)$ are continuous with respect to
$(t,r,x)\in \Lambda$. For any $(t_0,r_0,x_0) \in \Lambda$, fix
$\delta >0$ such that $t_0-\delta
> r_0+\delta$. The classical interior Schauder estimates (e.g.,
\cite[Thm.~3.5]{friedman})  imply that  for any $R>0$   there exists
a positive constant $C $ such that
\begin{equation}
\sup_{|t-t_0|\leq \delta, \, |r-r_0|\leq \delta}
 \|P(t,r)\varphi \|_{C^{ 2+\alpha}( B(x_0,R))} \le C\|\varphi \|_{\infty},
 \label{stima-KLL}
\end{equation}
for every $\varphi \in C_b(\R^d)$.   Applying
the interpolatory  estimates
\begin{equation}
\label{interp}
\|D_{i} \psi\|_{C(B(x_0,R)) }    \le
K_1\|\psi\|_{C(B(x_0,R)) } ^{\frac{1+\alpha}{2+\alpha}}
\|\psi\|_{C^{2+\alpha}( B(x_0,R) )}^{\frac{1}{2+\alpha}}, \quad i=1, \ldots, d,
\end{equation}
(which hold for every $\psi\in C^{2+\alpha}( B(x_0,R))$   and some
positive constant $K_1=K_1(\alpha, R)$, see e.g., \cite[Sect.~
4.5.2, Rem.~2]{Triebel}) to the function $\psi= P(t,r)f(r,\cdot)-
P(t_0,r_0)f(r_0,\cdot)$ with  $t\in [t_0-\delta, t_0+\delta]$,
$r\in[r_0-\delta, r_0+\delta] $, we   deduce
\begin{align*}
&\|D_iP(t,r)f(r,\cdot)- D_iP(t_0,r_0)f(r_0,\cdot)\|_{C(B(x_0,R)) }
\\
\le & K_1\|P(t,r)f(r,\cdot)- P(t_0,r_0)f(r_0,\cdot)\|_{C(B(x_0,R)) }
^{\frac{1+\alpha}{2+\alpha}}
\\
&\qquad \times (\|P(t,r)f(r, \cdot)\|_{C^{ 2+\alpha}( B(x_0,R))}
+\|P(t_0,r_0)f(r_0 ,\cdot ) \|_{C^{ 2+\alpha}( B(x_0,R))}
)^{\frac{1}{2+\alpha}}
\\
\le & K_1 \|P(t,r)f(r,\cdot)- P(t_0,r_0)f(r_0,\cdot)\|_{C(B(x_0,R))
} ^{\frac{1+\alpha}{2+\alpha}}
(2C\|f\|_{\infty})^{\frac{1}{2+\alpha}},
\end{align*}
where the last inequality follows from \eqref{stima-KLL}. Since
$(t,r,x)\mapsto P(t,r)f(r,\cdot)(x)$ is a continuous function in
$\Lambda$, the right-hand side vanishes as $(t,r)\to (t_0,r_0)$, and this implies that $D_iP(t,r)f(r,\cdot)(x)$ is continuous in $(t,r,x)\in \Lambda$.

Using the interpolatory estimates (\cite[Sect.~ 4.5.2, Rem.~2]{Triebel})
\begin{eqnarray*}
\|D_{ij} \psi\|_{C(B(x_0,R)) }    \le K_2\|\psi\|_{C(B(x_0,R)) } ^{\frac{ \alpha}{2+\alpha}}
\|\psi\|_{C^{2+\alpha}( B(x_0,R))}^{\frac{2}{2+\alpha}}, \quad i, j=1, \ldots, d,
\end{eqnarray*}
instead of \eqref{interp}, the same procedure yields that
$D_{ij}P(t,r)f(r,\cdot)(x)$ is continuous  in $\Lambda$ for any
$i,j=1,\ldots,d$.
 \end{proof}

The generator $G_{\infty}$ of $\T(t)$   in $C_{b}
(\U \times \R^d )$ may be defined through its resolvent. Namely, for
every $\lambda >0$, $D(G_{\infty})$ is the range of the operator
\begin{eqnarray*}
u\mapsto v(s,x):=  \int_{0}^{\infty} e^{-\lambda t}\T(t)u(s,x)dt ,
\end{eqnarray*}
and $G_{\infty}v = \lambda v -u$. The following result is taken from \cite[Prop. 6.3]{LZ}.

\begin{theorem}
Under Hypothesis \ref{hyp1} we have
\begin{eqnarray*}
D(G_{\infty})  = \bigg\{ f\in\bigcap_{q<\infty} W^{1,2}_{q,\rm loc}( \U \times \R^d, ds\times dx)
\cap C_{b}(\U \times \R^d ),\;\G f\in C_{b}(\U \times \R^d)\bigg\},
\end{eqnarray*}
where   ${\mathcal G}f (s,x) =  {\mathcal A}(s)f(s, \cdot)(x) -D_s f(s,x)$. Moreover, $D(G_{\infty}) $ coincides with the set of
the functions $u\in C_{b} (\U \times \R^d )$ such that
$\sup_{0<t\leq 1}t^{-1}\|\T(t)u - u\|_{\infty}<\infty$ and there
exists $g\in C_{b} (\U
\times \R^d )$ such that $t^{-1}(\T(t)u - u)\to g$, locally uniformly
in $\U \times \R^d$, as $t\to 0$.
\end{theorem}

For every  $T$-periodic evolution system of measures $\{ \nu_s:
s\in \R\}$ for $P(t,s)$, the measure $\nu (ds,dx) := \frac{1}{T} \,
\nu_s(dx) ds$ is invariant for $\T(t)$.  Indeed, the first part of
the proof of Corollary \ref{cor:continuity} shows that, for each
$\varphi\in C_b(\R^d)$, the function $s\mapsto \int_{\R^d}\varphi
\,d\nu_{s}$ is continuous in $\R$. Hence, for every Borel set
$\Gamma \subset \R^d$ the function $s\mapsto \nu_s(\Gamma )$ is
Lebesgue measurable, and $\nu$ is well defined. Moreover, for every
$f\in  C_{b} (\U \times  \R^{d })$ we have
\begin{align*}
\int_{(0,T)\times \R^{d}}\T (t)fd\nu & = \frac{1}{T} \int_0^Tds
\int_{\R^d}P(s,s-t)f(s-t,\cdot)d\nu_s
\\
& = \frac{1}{T}\int_0^Tds
\int_{\R^d}f(s-t,\cdot)d\nu_{s-t}
\\
& =  \frac{1}{T}\int_{-t}^{T-t}d\sigma
\int_{\R^d}f(\sigma ,\cdot)d\nu_{\sigma}
\\
& =  \int_{(0,T)\times \R^{d}}fd\nu ,
\end{align*}
where the last equality follows from the periodicity of the
function  $\sigma\mapsto
\int_{\R^d}f(\sigma ,\cdot)d\nu_{\sigma}$.

\begin{proposition}
\label{Pr:uniqueness} Under Hypothesis  \ref{hyp1}, $P(t,s)$ has a
unique  $T$-periodic evolution system of measures, and $\T(t)$ has a
unique invariant measure.
\end{proposition}
\begin{proof}
Let $\{ \nu_s: s\in \R\}$ be any $T$-periodic evolution system of
measures for $P(t, s)$. The arguments in \cite[Thm.~4.3]{DD} show
that $\nu = \frac{1}{T} \, \nu_s(dx)ds$ is ergodic for $\T(t)$,  in
the sense that, if $\Gamma $ is a Borel set in $\U\times \R^d$ such
that $\T(t)\one_{\Gamma} = \one_{\Gamma} $ for every $t$, then
either $\nu(\Gamma) =0$ or $\nu(\Gamma) = 1$. This implies that the
invariant measures $\mu$ and $\nu$ corresponding to two different
evolution systems of measures $\{ \mu_s: s\in \R\}$ and $\{ \nu_s:
s\in \R\}$, are either singular or coincide, see e.g.,
\cite[Prop.~3.2.5]{DPZ}.  But we know from \cite[Prop.~5.2]{KLL}
that $\mu_t$ and $\nu_t$ are equivalent to the Lebesgue measure in
$\R^d$ for every $t\in \R$, hence $\mu$ and $\nu$ are equivalent to
the Lebesgue measure in $\U\times \R^d$ so that they cannot be
singular.
Therefore, $\nu = \mu$, which implies $\nu_s = \mu_s$ for a.e. $s\in
\R$. Since $s\mapsto \int_{\R^d} \varphi (x)\nu_s(dx)$ and $s\mapsto
\int_{\R^d} \varphi (x)\mu_s(dx)$ are continuous for each
$\varphi\in C_b(\R^d)$ by Corollary \ref{cor:continuity}, then
$\nu_s = \mu_s$ for every $t\in \R$.

Let us prove that each invariant measure $\nu$ for $\T(t)$ comes
from an evolution system of measures. Arguing as in the case of time
depending Ornstein-Uhlenbeck operators (\cite[Prop.~4.2]{DPL}) we
see that  $\nu$  is of the type $\nu(ds,dx) = \frac{1}{T} \nu_s(dx)
dt$, where  $\{\nu_s: s\in \R\}$ is a family of  $T$-periodic
probability  measures. Let $f(s,x) = g(s)\varphi(x)$, with $g\in
C(\U)$ and $\varphi\in C_b(\R^d)$. For $t>0$ the equality  $\int_{\U
\times \R^d} \T(t)f\,d\nu =  \int_{\U \times \R^d}  f\,d\nu$ means
\begin{align*}
\int_{0}^{T} g(s) \int_{\R^d} \varphi(x)\nu_s(dx) ds & =
\int_{0}^{T} g(s-t) \int_{\R^d} P(s,s-t) \varphi(x)\nu_s(dx) ds
\\
& = \int_{-t}^{T-t} g( s ) \int_{\R^d} P(s + t,s) \varphi(x)\nu_{s
+t}(dx) ds\\
& = \int_{0}^{T}  g( s ) \int_{\R^d} P(s + t,s)\varphi(x)\nu_{s
+t}(dx) ds.
\end{align*}
Since $g$ is arbitrary, then $ \int_{\R^d} P(s + t,s)
\varphi(x)\nu_{s  +t}(dx) = \int_{\R^d} \varphi(x)\nu_s(dx)$ for
every $t>0$, which means that $\{ \nu_s: s\in \R\}$ is an evolution
system of measures.
\end{proof}

From now on we shall consider the  invariant measure $\mu$ for
$\T(t)$  defined by
\begin{eqnarray*}
\mu (ds,dx) := \frac{1}{T} \, \mu_s(dx) ds,
\end{eqnarray*}
where $\{\mu_s: t\in \R\}$ is  the unique $T$-periodic evolution
system of measures for $P(t,s)$.

As a consequence of  (\cite[p.~ 2067]{BKR}), there exists a continuous positive function $\rho: \U \times \R^d \mapsto \R$ such that $\mu (ds,dx) = \rho(s,x)ds\, dx$. The computation at the end of the proof of Proposition \ref{Pr:uniqueness} shows that the family of measures $\nu_s(dx) := \rho (s,x)dx$ are a $T$-periodic evolution system of measures. By uniqueness, $\nu_s = \mu_s/T$ for every $s$, i.e. the density of $\mu_s$ is $T\rho(s, \cdot)$ for every $s\in \R$.

For any $p\in [1,\infty)$, we introduce the space
$L^p (\U \times \R^{d},\mu )$ of all functions $f$ such that
$f(s+T,x)=f(s,x)$ for  a.e. $(s,x)\in  \R^{1+d}$
and
\begin{eqnarray*}
\|f\|_{L^p(\U \times \R^d, \mu)}^{p}:=\int_{(0,T)\times\R^d}|f|^p d\mu <\infty.
\end{eqnarray*}
We also use the symbol $\int_{\U \times\R^d}|f|^p d\mu $ for $\int_{(0,T)\times\R^d}|f|^p d\mu $.
If no confusion may arise, we write $\|f\|_{p} $ for $\|f\|_{L^p(\U \times \R^d, \mu)} $.

As all Markov semigroups having an invariant measure,  $\T(t)$ can
be extended to a semigroup of positive contractions in $L^p (\U
\times \R^d ,\mu )$ for any $p\in [1,\infty)$. We still call $\T
(t)$ these extensions, using the notation  $\T_p(t)$ only when we
deal with different $L^p$ spaces.

It is easy to see that  $\T(t)$ is strongly continuous in $L^p (\U \times
\R^d ,\mu )$ for   $1\leq p <\infty$. Indeed,   we   already
know that  $\T(t)f$ tends to $f$ locally
uniformly as $t\to 0^+$,  for any $f\in C_b(\U\times\R^d)$.
Moreover, $\|\T(t)f\|_{\infty}\le \|f\|_{\infty}$ for any $t>0$.
By dominated convergence, $\T(t)f$ tends to $f$ in
$L^p(\U\times\R^d,\mu)$ as $t\to 0^+$. Since $C_b(\U\times\R^d)$ is dense in
$L^p(\U\times\R^d,\mu)$, $\T(t)f$ tends to
$f$ in $L^p(\U\times\R^d,\mu)$ as $t\to 0^+$, for every $f\in
L^p(\U\times\R^d,\mu)$.

We denote by $G_p$ the  infinitesimal generator of $\T(t)$ in $L^p(\U \times \R^d ,\mu )$. In general,
the characterization of the domain $D(G_p)$ of $G_p$ is not obvious, and even
determining whether a given smooth function $f$ belongs to   $D(G_p)$ is not obvious. In the case of time depending
Ornstein-Uhlenbeck operators,    $D(G_2)$ has been
characterized in \cite{GL1} as the space of all $f\in L^2(\U \times
\R^d ,\mu )$ such that there exist $D_sf$, $D_if$, $D_{ij}f\in
L^2(\U \times \R^d ,\mu )$ for $i, j=1, \ldots, d$. A similar
characterization for $2\neq p\in (1,\infty)$ follows adapting to the
periodic case the procedure of \cite{GLS}. We do not expect the same
result in the general  case, since for functions $f$ with all
derivatives in $L^p(\U \times \R^d ,\mu )$, $\A(s)f(s,\cdot)$ does
not necessarily belong to $L^p(\R^{d}, \mu_t)$, even for bounded
diffusion coefficients. Fortunately, the explicit knowledge of
$D(G_p)$ is not necessary in several circumstances, provided we know
a good core of $G_p$. In the paper \cite{LZ} sufficient conditions
have been given for $C^{\infty}_{c}(\U \times \R^d )$ be a core of
$G_p$, for $1\le p<\infty$. In general, $C^{\infty}_{c}(\U \times
\R^d )$ is contained in $D(G_p)$ but it is not a core. However, we
have the following result (\cite[Thm.~6.7]{LZ}).

\begin{proposition}
\label{prop-2.9} Under Hypothesis \ref{hyp1}, $\T(t)$ maps
$D(G_{\infty})$ into itself, and $D(G_{\infty})$ is a core of $G_p$
for every $p\in [1,\infty)$ and every $t>0$.
\end{proposition}

Adapting to our situation a similar  result for evolution semigroups
in fixed Banach spaces (e.g., \cite[Thm. 3.12]{CL}), we determine
another core of  $G_p$.

To this purpose, for any $\tau \in \R$, $\chi \in
C^{\infty}_{c}(\R^d)$ and $\alpha \in C^{1}_{c} (\R)$  with $\supp
\alpha \subset (a, a +T)$ for some  $a\geq \tau$, we define the
function $u_{\tau, \chi, \alpha}:\R^{1+d}\to\R$, as  the
$T$-periodic (with respect to $s$) extension of the function
$(s,x)\mapsto\alpha(s)P(s,\tau)\chi(x) $ defined in $[a, a +T)\times
\R^d$.

\begin{proposition}
\label{Pr:core}
For each $\tau$,  $\chi  $ and $\alpha$ as
above, the function $u_{\tau, \chi, \alpha}$ belongs to $D(G_p)$ and
the linear span $\mathcal C$ of the functions $u_{\tau, \chi,
\alpha}$ is a core for $G_p$, for each $p\in [1,\infty)$.
\end{proposition}
\begin{proof}
Any function $u\in \mathcal C$ is in $C^{1,2}(  \R^{1+d} )$ by
(the proof of) \cite[Thm.~2.2]{KLL} and, since it is periodic in
time and bounded, it belongs to $L^p (\U \times \R^d ,\mu )$ for every $p\in [1,\infty)$.

Fix $\tau$, $\alpha$, $\chi$ as in the statement  and
define $u:= u_{\tau, \chi,\alpha}$.  For each $s\in [a , a +T)$, $x\in \R^d$ we have
\begin{eqnarray*}
\G u(s,x) = -\{\alpha '(s)P(s,\tau)\chi(x) + \alpha(s)\A(s)P(s,
\tau)\chi(x)\} + \alpha(s)\A(s)P(s, \tau)\chi(x),
\end{eqnarray*}
so that
\begin{eqnarray*}
\G u(s,x)  =  -\alpha'(s)P(s,\tau)\chi(x) , \quad s\in  [a , a
+T), \;x\in \R^d,
\end{eqnarray*}
and, for every $k\in \Z$,
\begin{eqnarray*}
\G u(s,x)  =  -\alpha '(s-kT)P(s-kT,\tau)\chi(x) , \quad s\in [a
+kT, a +(k+1)T), \;x\in \R^d.
\end{eqnarray*}
Let us prove that, for every $t>0$, $\T(t)u \in \mathcal C$. For
every $s\in \R$, let $k\in \Z$ be such that $s-t\in
 [a +kT, a+(k+1)T)$. Then $s-t-kT \geq \tau$, and for every  $x\in \R^d$ we have
\begin{align*}
(\T(t)u)(s,x) &=
\alpha (s-t-kT)P(s, s-t )P(s-t-kT,\tau )\chi(x) \\
&= \alpha (s-t-kT)P(s-kT,s-t-kT)P(s-t-kT,\tau ) \chi(x) \\
&= \alpha (s-t-kT)P(s-kT,\tau )   \chi(x),
\end{align*}
so that $\T(t)u $ is the $T$-periodic extension of the function
$(s,x)\mapsto \alpha(s-t)P(s,\tau)\chi(x)$ defined in  $[a +t, a+t
+T)\times\R^d$, which belongs to $ \mathcal C$. Moreover, $t\mapsto
\T(t)u $ is differentiable at $t=0$ with values in $L^p(\U \times
\R^d ,\mu )$ and we have
\begin{eqnarray*}
\left (\frac{d}{dt} \T(t)u\right )_{\big  |t=0} (s,x) = -
\alpha'(s)P(s, \tau) \chi(x) = \G u(s,x),\quad s\in [a, a +T),
\end{eqnarray*}
which shows that $u\in D(G_p)$ and $G_p u = \G u $.

Let us prove that $\mathcal C$ is dense in the domain of $G_p$.
Since $\T(t)$ maps $\mathcal C$ into itself for any $t>0$, it is
enough to prove that ${\mathcal C}$ is dense in $L^p (\U \times \R^d ,\mu )$.

We recall that the linear span of the functions $(s,x)\mapsto
\beta(s)\chi(x)$, with $\beta \in C^{1} (\U)$, $\chi \in
C^{\infty}_{c}(\R^d)$, is dense in  $L^p (\U \times \R^d ,\mu )$.
Therefore, it is enough to approximate any product $g = \beta \chi$ of
this type by elements of $\mathcal C$.

Fix $\beta \in C^{1} (\U)$, $\chi \in C^{\infty}_{c}(\R^d)$ and
$\eps >0$. Let $\tau \in \R$. Lemma 3.2 of \cite{KLL} implies that
$P(s,\tau)\chi$ tends to $\chi$, uniformly in $\R^d$, as $s\to
\tau^+$. Therefore, there exists $\tau'\in (\tau, \tau +T)$ such
that  $\|P(s,\tau)\chi -\chi\|_{\infty} \leq \eps$, for each $s\in
[\tau, \tau']$, which implies
\begin{eqnarray*}
\|P(s,\tau)\chi -\chi\|_{L^p(\R^d,\mu_s)} \leq \eps, \quad \tau\leq
s\leq \tau'.
\end{eqnarray*}
Let us cover $\mathbb T $ by a finite number of such
intervals ({\em mod} $T$)
$(\tau_k,\tau'_k)$, $k=1, \ldots, K$, and let
 $(\alpha_k)$ be an associated partition of   unity. Setting
\begin{eqnarray*}
u_k(s, x) =  \beta(s) \alpha_k(s)P(s, \tau_k)\chi(x), \quad s\in
[\tau_k, \tau_k +T)
\end{eqnarray*}
and still denoting by $u_k$ its $T$-periodic extension, the function $u$ defined by
\begin{eqnarray*}
u(s,x) :=  \sum_{k=1}^{K} u_k(s, x) , \quad s\in \R, \;x\in \R^d,
\end{eqnarray*}
belongs to $\mathcal C$, and we have
\begin{eqnarray*}
\|g(s, \cdot) - u(s, \cdot)\|_{L^p(\R^d,\mu_s)} \leq
\|\beta\|_{\infty} \sum_{k=1}^{K} \alpha_k(s) \|P(s,\tau_k)\chi
-\chi\|_{L^p(\R^d,\mu_s)} \leq \eps  \|\beta\|_{\infty},
\end{eqnarray*}
for any $s\in [0,T]$. Integrating with respect to  $s$ in $(0,T)$ we obtain
\begin{eqnarray*}
\|g-u\|_p \leq  \eps  \|\beta\|_{\infty},
\end{eqnarray*}
and the statement follows.
\end{proof}

\begin{corollary}
\label{restrizione} For $1<p<\infty$, $D(G_p)\subset W^{1,2}_{p,\rm
loc}(\U \times \R^{d},ds\times dx)$ and for every $r>0$ the restriction mapping
${\mathcal R}:D(G_p)\to W^{1,2}_{p }(\U \times B(0,r), ds\times
dx)$, defined by  ${\mathcal R}u= u_{|\U \times B(0,r)}$, is
continuous.
\end{corollary}
\begin{proof} Every $u
\in \mathcal{C}$ belongs to $C^{1,2} (\U \times \R^{d} )$, and
$G_pu(s,x)  = \A(s)u(s,x)  - D_s u(s,x) $, so that by classical
regularity results for parabolic equations in $L^p$ spaces with
respect to the Lebesgue measure there exists $C_1=C_1(r)$ such that
\begin{eqnarray*}
\|u\|_{W^{1,2}_{p }(\U \times B(0,r), ds\times dx)} \leq C_1(\|u \|_{L^p(\U \times B(0,2r), ds\times dx)} + \|G_pu \|_{L^p(\U \times B(0,2r), ds\times dx)}).
\end{eqnarray*}
On the other hand, since $\mu (ds,dx) =\rho(s,x)ds\,dx$ for a
positive continuous function $\rho$, there exists $C_2=C_2(r)$ such
that $\|\cdot \|_{L^p(\U \times B(0,2r), ds\times dx)} \leq C_2
\|\cdot \|_{L^p(\U \times \R^d, \mu)}$. Then, ${\mathcal R}$  is
continuous from $\mathcal C$ (endowed with the $D(G_p)$-norm) to
$W^{1,2}_{p }(\U \times B(0,r), ds\times dx)$, and since $\mathcal
C$ is dense in $D(G_p)$ the statement follows.
\end{proof}

The estimates on the space derivatives  of $P(t,s)f$ yield a useful
embedding result for  $D(G_p)$. We denote by
$W^{0,1}_{p}(\U \times \R^{d}, \mu)$ the set of the functions $f\in
L^{p} (\U \times \R^{d}, \mu)$ having space derivatives $D_if$ in
$L^{p} (\U \times \R^{d}, \mu)$, for every $i=1, \ldots, d$. It is a
Banach space with the norm
\begin{eqnarray*}
\|f\|_{W^{0,1}_{p}(\U \times \R^{d}, \mu)} = \|f\|_{L^{p} (\U \times \R^{d}, \mu)} +\sum_{i=1}^{d}
\|D_if\|_{L^{p} (\U \times \R^{d}, \mu)}.
 \end{eqnarray*}
Similarly, we denote by $C^{0,1}_{b}(\U \times \R^{d} )$ the set of
the functions $f\in C_b (\U \times \R^{d} )$ having space
derivatives $D_if$ in $C_b (\U \times \R^{d})$, for every $i=1,
\ldots, d$. It is a Banach space with the norm
 \begin{eqnarray*}
 \|f\|_{C^{0,1}_{b}(\U \times \R^{d} )} = \|f\|_{\infty} + \sum_{i=1}^{d} \|D_if\|_{\infty}.
 \end{eqnarray*}

\begin{proposition}
\label{thm-2.7} Assume that Hypotheses  \ref{hyp1}  and  \ref{hyp2}
are satisfied. Let $C_2$  be the constant  given by Theorem
\ref{Th:gradP(t,s)infinito}.  Then,
for every  $t>0$ and   $f\in C_b(\U\times\R^d)$ we have
\begin{equation}
\| \,|\nabla_x\T(t)f| \, \|_{\infty } \leq C_2 \max\{t^{-1/2}, \; 1\}
\|f\|_{ \infty } .
 \label{grad-infty-semigruppone>1}
\end{equation}
Moreover,  $D(G_{\infty})$ is continuously embedded into $C^{0,1}_{b}(\U \times \R^{d} )$.

If for some $p>1$ the constant $\ell_p$ in \eqref{ell_p} is finite, let $C_3$ be the constant in estimate \eqref{grad-punt-1}. Then, for every $f\in L^{p} (\U \times \R^{d}, \mu)$,
\begin{equation}
\| \,|\nabla_x\T(t)f| \, \|_p  \leq  \left\{ \begin{array}{ll}
 C_3 e^{\ell_p}  t^{-1/2}\|f\|_p, & 0<t\leq 1,
 \\[2mm]
 C_3  \min\{ e^{\ell_pt}, e^{\ell_p} \} \|f\|_p, & t\geq 1,
\end{array}\right.
 \label{grad-p-semigruppone>1}
\end{equation}
and $D(G_p)$ is continuously embedded into $W^{0,1}_{p}(\U \times \R^{d}, \mu)$.
Moreover,
\begin{equation}
\| \,|\nabla_x\T(t)f| \, \|_p  \leq  e^{\ell_pt} \| \,|\nabla_x f| \, \|_p, \quad t>0, \; f\in W^{0,1}_{p}(\U \times \R^{d}, \mu).
\label{grad-grad}
\end{equation}
\end{proposition}
\begin{proof}
Recalling that $\T(t)f(s,x) = P(s,s-t)f(s-t, \cdot)(x)$, estimate
\eqref{grad-infty-semigruppone>1} follows immediately from
\eqref{boh} and \eqref{star-luca}.

Let us   prove that  $D(G_{\infty})$ is continuously
embedded into $C^{0,1}_b(\R^{1+d})$.
  $D(G_{\infty})$ coincides with the range of the resolvent
$R(\lambda, G_{\infty})$, for any $\lambda > 0$. Since $R(\lambda,
G_{\infty})f(s,x) = \int_{0}^{\infty} e^{-\lambda t}\T(t)f
(s,x)\,dt$ for any $(s,x)\in\R^{1+d}$, estimate
\eqref{grad-infty-semigruppone>1} implies that the   derivatives
$D_iR(\lambda, G_{\infty})f $ are bounded by $C\|f\|_{\infty}$ for
some $C>0$. Their continuity follows from the continuity of the
space derivatives of $\T(t)f$ (Lemma \ref{le:derivateT(t)}) through
the dominated convergence theorem.

Assume now that $\ell_p<+\infty$ and let $f\in L^{p} (\U \times
\R^{d}, \mu)$. Using \eqref{grad-p} we obtain
\begin{align*}
\int_{\U \times \R^d} |\nabla_x\T(t)f |^p d\mu &=
\frac{1}{T} \int_{0}^{T} \int_{\R^d}|\nabla_x P(s,s-t)f(s-t,\cdot)(x)|^p  \mu_s(dx)\,ds \\
&\le   \frac{(C_3e^{\ell_p t})^p}{T} \max\{t^{-p/2},\, 1\}   \int_{0}^{T}  \|f(s-t, \cdot)\|_{L^p(\R^d, \mu_{s-t})}^{p}ds \\
&= (C_3e^{\ell_p t})^p  \max\{t^{-p/2},\, 1\}
\|f\|_{L^p(\U \times \R^d, \mu)}^{p}.
\end{align*}
If $\ell_p>0$,  for $t\geq 1$ we use \eqref{grad-p>1} instead of  \eqref{grad-p}, and we get
\begin{eqnarray*}
\|\,|\nabla_x \T(t)f|\,\|_p\leq C_3e^{\ell_p} \| f \|_p.
\end{eqnarray*}
In any case,  \eqref{grad-p-semigruppone>1} holds.
The embedding
$D(G_p)\subset W^{0,1}_{p}(\U \times \R^{d}, \mu)$ follows again
from the equality $R(\lambda, G_p)f
= \int_{0}^{\infty} e^{-\lambda t}\T(t)f\,dt$, for any $\lambda >
0$.

Estimate \eqref{grad-grad} follows from Theorem \ref{Th:gradP(t,s)puntuale}(i)  and from the density of $C^{0,1}_{b}(\U \times \R^{d}, \mu)$ in $W^{0,1}_{p}(\U \times \R^{d}, \mu)$.
\end{proof}

If some bounds on $Q$ and on $\langle b,x\rangle$ hold,  we can prove
important integration formulae in $D(G_{\infty})$. They yield the embeddings $D(G_p)\subset W^{0,1}_{p}(\U \times \R^{d}, \mu)$, even without the assumption $\ell_p <\infty$.

\begin{proposition}
\label{lemma1 parte1}
Assume that Hypotheses  \ref{hyp1}  and  \ref{hyp2} are satisfied.  Then:
\begin{enumerate}[\rm (a)]
\item
If  there exists $C>0$ such that
\begin{equation}
\hbox{\hskip .6cm}\| Q(s,x)\|_{{\mathcal L}(\R^d)} \leq C(|x|+1) V(x), \;\; \langle b(s,x), x\rangle \leq C(|x|^2+1)V(x), \quad (s,x)\in\R^{1+d},
\label{stimapiu'}
\end{equation}
then for every $p \in (1,\infty)$ and   $u \in D(G_{\infty})$,  $  |u|^{p-2} \langle Q \nabla_x u, \nabla_x u
\rangle\chi_{\{u\neq 0\}}$
belongs to $L^1 (\U \times \R^d,\mu )$, and
\begin{equation}
 \label{integration<Hu}
 \int_{\U \times \R^d }  |u|^{p-2} \langle Q \nabla_x u, \nabla_x u
\rangle\chi_{\{u\neq 0\}} \,d \mu \leq  - \frac{1}{p-1} \int_{\U \times \R^d } u |u|^{p-2}G_p u \, d \mu;
 \end{equation}
\item
if there exists $C>0$ such that
\begin{equation}
\hbox{\hskip .6cm}\| Q(s,x)\|_{{\mathcal L}(\R^d)} \leq C(|x|+1) V(x), \;\; |\langle b(s,x), x\rangle | \leq C(|x|^2+1) V(x), \quad (s,x)\in \R^{1+d},
\label{stimamodulo}
\end{equation}
then for  $p\geq 2$ and $u \in D(G_{\infty})$ we have
\begin{equation}
 \label{integration Hu}
 \int_{\U \times \R^d }  |u|^{p-2} \langle Q \nabla_x u, \nabla_x u
\rangle\chi_{\{u\neq 0\}} \,d \mu = - \frac{1}{p-1} \int_{\U \times \R^d } u |u|^{p-2}G_p u \, d \mu.
 \end{equation}
\item
If the diffusion coefficients $q_{ij}$ are bounded, then for every $p\in (1, \infty)$
and $u \in D(G_{\infty})$, \eqref{integration Hu} holds.
\end{enumerate}
\end{proposition}
\begin{proof}
{\em Step 1: $p\ge 2$.}
Let $u\in D(G_{\infty})$. Then,
\begin{eqnarray*}
\G (|u|^p) = pu|u|^{p-2}\G u + p(p-1)|u|^{p-2}  \langle Q \nabla_x u, \nabla_x u \rangle .
\end{eqnarray*}
Since   $u$, $D_iu$, $\G u$ are bounded, if the diffusion
coefficients are bounded $\G (|u|^p)$ is bounded too. Therefore,
$|u|^p \in D(G_{\infty})\subset D(G_1)$, so that  $ \int_{\U \times
\R^d}  \G (|u|^p)   d\mu =0$ and \eqref{integration Hu} holds.

If  \eqref{stimapiu'}  holds,  the mapping $(s,x)\mapsto \| Q(s,x)\|_{{\mathcal L}(\R^d)}$ is in $L^1(\mathbb T\times\R^d,\mu)$, so that $\G (|u|^p)$  belongs to $L^1(\U \times \R^d, \mu)$. This does not guarantee that $|u|^p \in D(G_1)$ in the case of unbounded diffusion coefficients. However, we shall show that $ \int_{\U \times \R^d}  \G (|u|^p)   d\mu \leq 0$, and that $ \int_{\U \times \R^d}  \G (|u|^p)   d\mu = 0$ if the stronger condition  \eqref{stimamodulo} holds.

Let $\eta\in C^{\infty}(\R)$ be a nonincreasing function such that
$\one_{[0,1]}\le \eta\le\one_{[0,2]}$. For $R\geq 1$ define the
functions $\theta_R(x):=\eta(|x|/R)$ for any $x\in \R^d$. Since $\G
(|u|^p)\in L^1(\U \times \R^d, \mu)$, then
\begin{equation}
 \label{limiteu^p}
 \int_{\U \times \R^d}  \G (|u|^p)   d\mu  =  \lim_{R\to \infty}  \int_{\U \times \R^d}\G(|u|^p)  \theta_R d\mu .
\end{equation}
Let us  estimate the integrals $ \int_{\U \times \R^d}  \G(|u|^p)  \theta_R d\mu $.
For every $R>0$, $|u|^p\theta_R $ belongs to
$\bigcap_{p<\infty}W^{1,2}_{\rm p,\rm loc}(\R^{1+d})\cap C_b(\mathbb T\times\R^d)$ and
\begin{eqnarray*}
\G(|u|^p\theta_R) =  |u|^p \G(\theta_R) +   \G(|u|^p) \theta_R +
2pu|u|^{p-2} \langle Q\nabla_x u,\nabla \theta_R\rangle
\end{eqnarray*}
belongs to $C_b(\U\times\R^d)$. Hence, $|u|^p\theta_R \in D(G_{\infty})$, so that the mean value of $\G(|u|^p\theta_R) $ vanishes. This means
\begin{equation}
 \label{approssimata}
  \int_{\U \times \R^d}  \G(|u|^p) \theta_R \,d\mu =
\int_{\U \times \R^d} (- |u|^p \G (\theta_R)  -
2pu|u|^{p-2}\langle Q\nabla_xu, \nabla \theta_R\rangle )d\mu .
 \end{equation}

Let us compute $\G(\theta_R)(s,x)  = ( \A (s)\theta_R)(s,x)$. We have $D_j\theta_R(0)= D_{ij}\theta_R(0) =0$ and
\begin{align*}
&D_j\theta_R(x)=\eta'\left (\frac{|x|}{R}\right )\frac{x_j}{|x|R},\\
&D_{ij}\theta_R(x)=\eta''\left (\frac{|x|}{R}\right
)\frac{x_ix_j}{|x|^2R^2}+\eta'\left (\frac{|x|}{R}\right )
\frac{\delta_{ij}}{|x|R}-\eta'\left (\frac{|x|}{R}\right
)\frac{x_ix_j}{|x|^3R},
\end{align*}
for any $x\in\R^d\setminus \{0\} $ and any $i,j=1,\ldots,d$.
Therefore,
\begin{align*}
(\A (s) \theta_R )(s,x)=& \; \eta''\left (\frac{|x|}{R}\right )\frac{\langle Q(s,x)x,x\rangle}{|x|^2R^2}+
\eta'\left (\frac{|x|}{R}\right )\frac{\tr (Q(s,x))}{|x|R}\\
&-\eta'\left (\frac{|x|}{R}\right )\frac{\langle Q(s,x)x,x\rangle}{|x|^3R}
+\eta'\left (\frac{|x|}{R}\right )\frac{\langle b(s,x),x\rangle}{|x|R}.
\end{align*}
Since $\eta'(r)$ and
$\eta''(r)$ vanish if $r\notin (1,2)$, there exists $C_1>0$ such  that
\begin{eqnarray*}
\left |\eta''\left (\frac{|x|}{R}\right )\frac{\langle
Q(s,x)x,x\rangle}{|x|^2R^2}+ \eta'\left (\frac{|x|}{R}\right
)\frac{\tr (Q(s,x))}{|x|R}-\eta'\left (\frac{|x|}{R}\right
)\frac{\langle Q(s,x)x,x\rangle}{|x|^3R}\right |\le
\frac{C_1V(x)}{R},
\end{eqnarray*}
for any $(s,x)\in\R^{1+d}$. Moreover, $\langle Q\nabla_xu, \nabla
\theta_R\rangle$ goes to $0$ pointwise as $R\to \infty$, and for
each $  s\in \R$ and $x\in \R^d$ we have
\begin{eqnarray*}
|  \langle Q(s,x)\nabla_x u(s,x), \nabla \theta_R(x)\rangle | \leq C(|x|+1)V(x) \|\,|\nabla_xu|\,\|_{\infty}\frac{1}{R}\bigg| \eta'\left(\frac{|x|}{R}\right) \bigg|
\leq 3C\|\eta '\|_{\infty} V(x).
\end{eqnarray*}
Thus, for $ (s,x)\in\U\times (\R^d\setminus\{0\})$ we have
\begin{equation}
- (|u |^p \G (\theta_R) + 2pu|u|^{p-2} \langle Q\nabla_xu, \nabla \theta_R\rangle)(s,x)   = f_R(s,x) - \eta'\left (\frac{|x|}{R}\right )\frac{\langle b(s,x),x\rangle}{|x|R},
\label{Atheta}
\end{equation}
where $\lim_{R\to \infty} \|f_R\|_{L^1(\U\times\R^d,\mu)} =0$. Let us split  $\langle b(s,x),x\rangle = \langle b(s,x),x\rangle^+ -
\langle b(s,x),x\rangle^-$. Then $ \eta'(|x|/R )\frac{\langle b(s,x),x\rangle^-}{|x|R} \leq 0$, while
$ \eta'(|x|/R ) \frac{\langle b(s,x),x\rangle^+}{|x|R}$ goes to $0$ pointwise as $R\to \infty$, and by \eqref{stimapiu'}
\begin{equation}
\bigg| \eta'\left (\frac{|x|}{R}\right )\frac{\langle b(s,x),x\rangle^+}{|x|R}\bigg|  \leq \bigg| \eta'\left (\frac{|x|}{R}\right )\bigg|  \frac{C(|x|^2+1)V(x)}{|x| R}
\leq 5C \|\eta '\|_{\infty} V(x),
\label{infinit}
\end{equation}
for any $(s,x)\in\R^{1+d}$, so that by dominated convergence
\begin{equation}
\lim_{R\to \infty}  \int_{\U \times \R^d}  \eta'\left (\frac{|x|}{R}\right ) \frac{\langle b(s,x),x\rangle^+}{|x|R} d\mu =0.
\label{infinit1}
\end{equation}
Formulae  \eqref{approssimata} and  \eqref{Atheta} yield, for every $R>0$,
\begin{equation}
\label{infinit2}
 \int_{\U \times \R^d}\G(|u|^p) \theta_Rd\mu \leq  \int_{\U \times \R^d}\bigg( f_R(s,x) -  \eta'\left (\frac{|x|}{R}\right ) \frac{\langle b(s,x),x\rangle^+}{|x|R} \bigg) d\mu,
\end{equation}
where the right-hand side goes to $0$ as  $R\to \infty$. Now, taking \eqref{limiteu^p} into account,  \eqref{integration<Hu} follows.
If in addition \eqref{stimamodulo} holds, estimate \eqref{infinit} and its consequence   \eqref{infinit1} hold
 with  $\langle b(s,x),x\rangle^+$ replaced by $\langle b(s,x),x\rangle$, so that    \eqref {infinit2} may be replaced by
\begin{eqnarray*}
\int_{\U \times \R^d}\G(|u|^p) \theta_Rd\mu  =   \int_{\U \times \R^d}\bigg( f_R(s,x) -  \eta'\left (\frac{|x|}{R}\right ) \frac{\langle b(s,x),x\rangle}{|x|R} \bigg) d\mu
\end{eqnarray*}
and letting $R\to \infty$,  \eqref{limiteu^p} implies \eqref{integration Hu}.

 \vspace{2mm}
 \noindent {\em Step 2: $1<p<2$.}
  Fix $u \in D(G_{\infty})$ and $\delta >0$. Then, the function $u_{\delta} :=
(u^2 + \delta)^{\frac{p}{2}} - \delta^{\frac{p}{2}}$ is bounded and
continuous in $\U \times \R^d $. A straightforward computation shows
that
\begin{eqnarray*}
\G u_{\delta} = p u (u^2 + \delta)^{\frac{p}{2}-1} \G u +
p(u^2+\delta)^{\frac{p}{2}-2}[(p-1)u^2+\delta] \langle Q \nabla_x u,
\nabla_x u  \rangle.
\end{eqnarray*}
If \eqref{stimapiu'} holds, then $\G u_{\delta}$ belongs to
$L^1(\U\times\R^d,\mu)$. Moreover,
\begin{eqnarray*}
\G(u_{\delta}\theta_R) =  u_{\delta}\G (\theta_R) +   \G(u_{\delta})
\theta_R + 2pu(u^2+\delta)^{\frac{p-2}{2}} \langle Q\nabla_x
u,\nabla \theta_R\rangle
\end{eqnarray*}
belongs to $C_b(\U\times\R^d)$, and
\begin{eqnarray*}
|u(u^2+\delta)^{\frac{p-2}{2}} \langle Q\nabla_x u,\nabla
\theta_R\rangle|\leq\frac{2C\|\eta'\|_{\infty}}{R}V
(\|u\|_{\infty}^2+\delta)^{\frac{p-1}{2}}\|\,|\nabla_xu|\, \|_{\infty}.
\end{eqnarray*}
The   arguments used in Step 1 show that $\int_{\U\times\R^d}\G u_{\delta} d\mu\le 0$,
i.e.
\begin{equation}
\int_{\U\times\R^d}(u^2+\delta)^{\frac{p}{2}-2}[(p-1)u^2+\delta]
\langle Q \nabla_x u, \nabla_x u  \rangle d\mu\le
-\int_{\U\times\R^d}u (u^2 + \delta)^{\frac{p}{2}-1} \G u\, d\mu.
\label{luca-p<2}
\end{equation}
If the diffusion coefficients are bounded, $\G u_{\delta}$ belongs to  $C_b(\U\times\R^d)$, hence $u_{\delta}\in D(G_{\infty})\subset D(G_1)$ and the inequality in \eqref{luca-p<2} can be replaced by an equality.

By dominated convergence, the right-hand side of \eqref{luca-p<2} converges, as $\delta \to 0$, to the corresponding
integral  with $\delta =0$. Indeed, for any $\delta >0$, the function
$|u| (u^2 + \delta)^{\frac{p}{2}-1}|G_pu|$ is bounded from above by
$|u|^{p-1}|G_p u|$ since $p < 2$, and the $\mu $-a.e. pointwise
convergence is obvious. On the other hand, the functions
\begin{eqnarray*}
(u^2 + \delta)^{\frac{p}{2}-2}
[(p-1)u^2 + \delta]\langle Q \nabla_x u, \nabla_x u \rangle
\end{eqnarray*}
converge pointwise
a.e. (with respect to the Lebesgue measure) in $\{u \neq 0\}$ to the function $ |u|^{p-2} \langle Q \nabla_x u, \nabla_x u\rangle\chi_{\{u\neq 0\}}$,
and $\nabla_x u =0$ a.e. (with respect to the Lebesgue measure) in the set
 $\{u = 0\}$. Since $\mu $ is absolutely continuous with respect
to the Lebesgue measure, it follows that $(u^2 +
\delta)^{\frac{p}{2}-2}[(p-1)u^2 + \delta] \langle Q \nabla_x u,
\nabla_x u \rangle$ converges to 0 $\mu $-a.e. in $\{u =0\}$. This
shows that $(u^2 + \delta)^{\frac{p}{2}-2}[(p-1)u^2 + \delta]
\langle Q \nabla_x u, \nabla_x u \rangle$ converges pointwise to $
|u|^{p-2} \langle Q \nabla_x u, \nabla_x u\rangle\chi_{\{u\neq 0\}}$
$\mu $-a.e. in $\U \times \R^d $. By the Fatou Lemma,
\begin{align*}
  \int_{{\{u\neq
0\}} } |u|^{p-2} \langle Q \nabla_x u, \nabla_x u \rangle
d \mu
\leq & \liminf_{\delta \to 0^+} \int_{\U \times \R^d } (u^2 +
\delta)^{\frac{p}{2}-2}
[(p-1)u^2 + \delta] \langle Q \nabla_x u, \nabla_x u \rangle d \mu\\
= & - \frac{1}{p-1} \lim_{\delta \to 0^+} \int_{\U \times \R^d } u
(u^2 + \delta)^{\frac{p}{2}-1}
G_p u\, d \mu\\
 = & - \frac{1}{p-1} \int_{\U \times \R^d } u |u|^{p-2}G_p  u \,d \mu ,
\end{align*}
and this implies that $ |u|^{p-2} \langle Q \nabla_x u, \nabla_x u\rangle\chi_{\{u\neq 0\}}$ belongs to $L^1(\U \times \R^d,\mu)$.
Now, since
\begin{eqnarray*}
(u^2 + \delta)^{\frac{p}{2}-2}[(p-1)u^2 + \delta] \langle Q \nabla_x u, \nabla_x u \rangle
\leq 2 (4-p)^{\frac{p-4}{2}}  |u|^{p-2} \langle Q \nabla_x u, \nabla_x u\rangle\chi_{\{u\neq 0\}},
\end{eqnarray*}
we can apply the dominated convergence theorem
to both sides of \eqref{luca-p<2} (which is an equality if the diffusion coefficients are bounded) and conclude that \eqref{integration<Hu}
 holds   if \eqref{stimapiu'} is satisfied, and  that \eqref{integration Hu}
 holds   if the diffusion coefficients are bounded.
 \end{proof}


\begin{corollary}
\label{cor:2}
Let Hypotheses  \ref{hyp1}  and  \ref{hyp2} hold.  Then:
\begin{enumerate}[\rm (a)]
\item
if the diffusion coefficients are bounded, or if \eqref{stimapiu'} holds, $D(G_p) \subset W^{0,1}_{p}(\U \times \R^d)$ for each $p\in (1,\infty)$, and the mapping $f\mapsto Q^{1/2}\nabla_xf$ is continuous from $D(G_p)$ into $(L^{ p}(\U \times \R^d,\mu))^d$ for $1<p\leq 2$;
\item
if the diffusion coefficients are bounded, or if \eqref{stimapiu'} holds, inequality \eqref{integration<Hu} holds for  every $p\in (1,\infty)$ and $u \in D(G_p)$;
\item
if the diffusion coefficients are bounded,  equality \eqref{integration Hu} holds for  every $p\geq 2$
and $u \in D(G_p)$.
\end{enumerate}
\end{corollary}
\begin{proof}
(a).  Let $1<p\leq 2$ and $f\in D(G_{\infty})$. Using the H\"older
inequality,  \eqref{integration<Hu} and then the H\"older inequality
again,  we get
\begin{align*}
\left( \int_{\U\times\R^d}|Q^{1/2}\nabla_xf|^pd\mu
\right)^{\frac{2}{p}} &\leq\left (\int_{\U\times\R^d}|f|^pd\mu\right
)^{\frac{2}{p}-1 }
\int_{\U\times\R^d}|f|^{p-2}|Q^{1/2}\nabla_xf|^2\chi_{\{f\neq
0\}}d\mu
\\
&\leq \|f\|_{p}^{2-p}  \frac{1}{p-1}
\int_{\U\times\R^d}|f|^{p-1}|G_pf|d\mu
\\
&\leq  \frac{1}{p-1}  \|f\|_{p} \|G_pf\|_{p}.
\end{align*}
Therefore, $\| \,| Q^{1/2} \nabla_xf |\,\|_p \leq 1/(2\sqrt{p-1}) \|f\|_{D(G_p)}$. Since $D(G_{\infty})$ is dense in $D(G_p)$, the mapping  $f\mapsto Q^{1/2} \nabla_xf  $ is bounded from $D(G_p)$ to $(L^{ p}(\U \times \R^d, \mu))^d$ and, since $Q(s,x) \geq \eta_0 I$ for each $s$ and $x$, also the mapping $f\mapsto \nabla_xf $ is  bounded from $D(G_p)$ to $(L^{ p}(\U \times \R^d, \mu))^p$, that is $D(G_p) \subset W^{0,1}_{p}(\U \times \R^d, \mu)$.

If $p>2$, the embedding follows by interpolation  between $L^2$ and
$L^{\infty}$. Precisely, since for $i=1, \ldots, d$ and $\lambda >0$
the mappings $f \mapsto D_i \int_{0}^{\infty} e^{-\lambda t} \T(t)f
\,dt  $ are bounded in  $L^{ 2}(\U \times \R^d, \mu)$ by statement
(a), and in $L^{ \infty}(\U \times \R^d, \mu)$ by Theorem
\ref{thm-2.7}, they are bounded in $L^{ p}(\U \times \R^d, \mu)$ for
every $p\in (2, \infty)$. On the other hand, $\int_{0}^{\infty}
e^{-\lambda t} \T(t)f \,dt= R(\lambda, G_p)f$ for every $f\in L^{
p}(\U \times \R^d, \mu)$. Therefore, the range of  $R(\lambda,
G_p)$, which is the domain of $G_p$, is continuously embedded in
$W^{0,1}_{p}(\U \times \R^d, \mu)$.

\vspace{2mm} (b). Consider  the nonlinear functions on $D(G_p)$
defined by
\begin{eqnarray*}
H(u) = |u|^{p-2} \langle Q \nabla_x u, \nabla_x u
\rangle\chi_{\{u\neq 0\}}, \quad K(u) =   u |u|^{p-2}G_p u.
\end{eqnarray*}
It is easy to see that $K$ is continuous with values in $L^1(\U
\times \R^d,\mu)$. Concerning $H$, fix $u\in D(G_p)$ and let
$(u_n)\subset  D(G_{\infty})$ converge to $u$ in $D(G_p)$. By
statement (a), $(u_n)$ converges to $u$ in
$W^{0,1}_p(\U\times\R^d,\mu)$. We may assume (possibly replacing
$u_n$ by a suitable subsequence) that $u_n$, $D_iu_n$ converge,
respectively, to $u$, $D_iu$ pointwise $\mu$-a.e, $i=1, \ldots, d$,
so that $H(u_n)$ converges to $H(u)$ pointwise $\mu$-a.e in $\{u\neq 0\}$. For every
$n\in \N$, $u_n$ satisfies \eqref{integration<Hu} by Proposition
\ref{lemma1 parte1}. Letting $n\to \infty$ we get, by the Fatou
Lemma,
\begin{align*}
\int_{\U \times \R^d }  |u|^{p-2} \langle Q \nabla_x u, \nabla_x u
\rangle\chi_{\{u\neq 0\}} \,d \mu  \leq &   \liminf_{n\to  \infty}
\int_{\U \times \R^d }  |u_n|^{p-2} \langle Q \nabla_x u_n, \nabla_x
u_n \rangle\chi_{\{u_n\neq 0\}} \,d \mu
\\
\leq  &
\lim_{n\to \infty}  \int_{\U \times \R^d } u_n |u_n|^{p-2}G_p u_n \, d \mu
\\
= & \int_{\U \times \R^d } u |u|^{p-2}G_p u \, d \mu ,
\end{align*}
that is,  \eqref{integration<Hu} holds for every $u\in D(G_p)$.

\vspace{2mm}
(c). If the diffusion coefficients are bounded, using statement (a) and the H\"older inequality it is easy to see that the
function $H : D(G_p)\to L^1(\U \times \R^d,\mu)$ is
continuous for $p\geq 2$.  Since   $D(G_{\infty})$ is dense in $D(G_p)$ and
\eqref{integration Hu} holds  for $u\in D(G_{\infty})$ by
Proposition \ref{lemma1 parte1}, then it holds for $u\in D(G_p)$.
 \end{proof}

As in the case of  evolution semigroups in fixed Banach spaces, the
spectral mapping theorem holds for $\T(t)$.
The proof is the same of \cite[Prop.~2.1]{GL2} with obvious changes, and it is omitted.

\begin{theorem}
\label{Th:SMT}
Let $1\leq p<\infty$, and denote by $\T _p(t)$ the
realization of $\T(t)$ in $L^p(\U \times \R^d, \mu)$. Then we have
\begin{eqnarray*}
\sigma (\T _p(t)) \setminus \{0\} = e^{t\sigma (G_p)}, \quad t>0.
\end{eqnarray*}
\end{theorem}

\section{Asymptotic behavior}
\setcounter{equation}{0}

In this section we prove some asymptotic behavior results for
$\T(t)$ that yield asymptotic behavior results for $P(t,s)$.

We introduce a projection $\Pi$ on functions depending only on time, defined by
\begin{eqnarray*}
\Pi f(s,x): = m_s f(s, \cdot) = \int_{\R^d} f(s,y)\mu_s(dy), \quad s\in \U, \;x\in \R^d.
\end{eqnarray*}
It is easy to see that $\|\Pi\|_{{\mathcal L}(C_b(\U \times \R^d))}
= 1$, and $\|\Pi\|_{{\mathcal L}(L^p(\U \times \R^d, \mu))} = 1$.
The ranges of $\Pi(C_b(\U \times \R^d))$ and of $\Pi (L^p(\U \times
\R^d, \mu))$ may be identified with $C(\U)$ and with $L^p(\U ,
\frac{ds}{T})$, respectively. $\T(t)$ leaves $C(\U)$ and   $L^p(\U ,
\frac{ds}{T})$  invariant, and the part of
$\T(t)$ in such spaces is just the translation semigroup $f\mapsto
f(\cdot -t)$. Although $\T(t)$  is not strongly continuous in
$C_b(\U \times \R^d)$, the part of $\T(t)$ in $C(\U)$ is strongly
continuous. The infinitesimal generators of the parts of $\T(t)$ in
$C(\U)$ and in $L^p(\U ,\frac{ds}{T})$ have domains (isomorphic to)
$C^1(\U)$ and $W^{1,p}(\U,\frac{ds}{T})$,
respectively, and coincide with $-D_s$.

\vspace{2mm}

In the next theorems  we relate the asymptotic behavior of $\T(t) $
to  the asymptotic behavior  of $P(t,s)$.

\begin{theorem}
\label{Pr:equivalenza}
Let Hypothesis \ref{hyp1} hold. For $1\leq p < \infty$, consider the following statements:
\begin{enumerate}[\rm (i)]
\item
for each $f\in  L^p(\U \times \R^d, \mu)$ we have
\begin{equation}
\label{CompAsT(t)}
\lim_{t\to \infty} \| \T(t)(f-\Pi f) \|_{L^p(\U \times \R^d, \mu)} =0;
\end{equation}
\item
for each $ \varphi  \in C_b(\R^d)$ we have
\begin{equation}
\label{compas_indietro} \exists/\forall t\in \R, \quad  \lim_{s\to -
\infty} \| P(t,s)\varphi - m_s\varphi\|_{L^p(\R^d, \mu_t)} =0;
\end{equation}
\item
for some/each  $s\in \R$   we have
\begin{equation}
\label{compas} \lim_{t\to \infty} \| P(t,s)\varphi -
m_s\varphi\|_{L^p(\R^d, \mu_t)} =0, \quad  \varphi  \in L^p(\R^d,
\mu_s);
\end{equation}
\item
for  each $ \varphi  \in C_b(\R^d)$   we have
\begin{equation}
\label{compas_indietroR} \exists /\forall t \in \R, \quad \lim_{s\to
- \infty} \| P(t,s)\varphi - m_s\varphi\|_{L^{\infty}(B(0,R))} =0,
\quad   R>0;
\end{equation}
\item
for some/each $s\in \R$   we have
\begin{equation}
\label{compasR} \lim_{t\to \infty} \| P(t,s)\varphi -
m_s\varphi\|_{L^{\infty}(B(0,R))} =0, \quad     \varphi  \in
C_b(\R^d), \;R>0.
\end{equation}
\end{enumerate}
For every $p\in [1, \infty)$, statements (i), (ii), (iii) are equivalent, and they are implied by statements (iv) and (v).  If in addition Hypothesis \ref{hyp2} holds, for every $p\in [1, \infty)$ statements (i) to (v) are equivalent.
 \end{theorem}
\begin{proof}
The proof is split in several steps.

\vspace{2mm} \noindent {\em Step 1: $\exists /\forall$ parts of
statements (ii) to (v). } To begin with, let us consider statement
(ii). Let  $ \varphi  \in C_b(\R^d)$ and  $t_0\in \R$ be such that $
\lim_{s\to - \infty} $ $\| P(t_0,s)\varphi -
m_s\varphi\|_{L^p(\R^d,\mu_{t_0})}$ $=0$. Then, for $t>t_0$, we have
$P(t,s)\varphi - m_s\varphi = P(t,t_0)(P(t_0,s)\varphi - m_s\varphi
)$ so that $ \| P(t,s)\varphi - m_s\varphi\|_{L^p(\R^d, \mu_t)} \leq
\|P(t_0,s)\varphi - m_s\varphi\|_{L^p(\R^d, \mu_{t_0})}$, which goes
to $0$ as $s\to -\infty$. For $t<t_0$ fix $k\in \N$   such that
$t+kT\geq t_0 $. Then, $P(t,s)\varphi - m_s\varphi =
P(t+kT,s+kT)\varphi - m_{s+kT}\varphi $, and $\mu_{t} = \mu_{t+kT}$,
so that $ \| P(t,s)\varphi - m_s\varphi\|_{L^p(\R^d, \mu_t)}$ $ = $
$\|P(t+kT,s+kT)\varphi - m_{s+kT}\varphi \|_{L^p(\R^d, \mu_{t+kT})}$
vanishes  as $s\to -\infty$ by the first part of the proof.

The same arguments yield the $\exists/\forall$ part of statement
(iv). Indeed, let $ \varphi  \in C_b(\R^d)$ and  $t_0\in \R$ be such
that  $ \lim_{s\to - \infty} $ $\|P(t_0,s)\varphi -
m_s\varphi\|_{L^{\infty}(B(0,R))}$ $=0$ for each $R>0$. For $t>t_0$
we have $P(t,s)\varphi - m_s\varphi = P(t,t_0)\varphi_s$, where
$\varphi_s := P(t_0,s)\varphi -  m_s\varphi$ goes to $0$ locally
uniformly as $s\to -\infty$. Corollary \ref{cor:continuity}(b)
yields $\lim_{s\to -\infty}\sup_{t >t_0}
\|P(t,t_0)\varphi_s\|_{L^{\infty}(B(0,R))}$ $=0$ for each $R>0$, that
is \eqref{compas_indietroR} holds for $t>t_0$ (even uniformly with
respect to $t$). If $t<t_0$ it is sufficient to fix  $k\in \N$ such
that $t+kT\geq t_0 $ and to argue as above.

Concerning statement (iii), if   \eqref{compas} holds for $s=s_0$,
then it holds for each $s\in \R$. Indeed, for $s<s_0$ and $ \varphi
\in L^p(\R^d, \mu_s)$ we have
$P(t,s)\varphi=P(t,s_0)P(s_0,s)\varphi$ and
\begin{eqnarray*}
m_{s_0}P(s_0,s)\varphi = \int_{\mathbb R^d}P(s_0,s)\varphi \,
d\mu_{s_0}= \int_{\mathbb R^d}\varphi \,d\mu_s=m_s \varphi ,
\end{eqnarray*}
so that
\begin{eqnarray*}
\|P(t,s)\varphi-m_s\varphi\|_{L^p(\R^d, \mu_t)}=
\|P(t,s_0)\psi-m_{s_0}\psi\|_{L^p(\R^d, \mu_t)},
\end{eqnarray*}
with   $\psi=P(s_0,s)\varphi$. Since $\psi \in L^p(\R^d,
\mu_{s_0})$, the right-hand side
  vanishes as $t\to \infty$.

For $s>s_0$   fix $k\in\mathbb N$ such that
$s-kT<s_0$. Then,
\begin{eqnarray*}
\|P(t,s)\varphi-m_s\varphi\|_{L^p(\R^d, \mu_t)}=
\|P(t-kT,s-kT)\varphi-m_{s-kT}\varphi\|_{L^p(\R^d, \mu_{t-kT})} .
\end{eqnarray*}
Since $s-kT<s_0$, by the first part of the proof the right-hand side  vanishes as $t\to\infty$.

The same arguments show that if \eqref {compasR} holds for some $s_0$, then it holds for each $s\in \R$.

\vspace{2mm}

\noindent{\em Step 2:  (i)   implies (ii). } Let us fix $\varphi \in
C^{\infty}_{c}(\R^d)$. By Step 1, it is enough to show that
$\lim_{s\to\infty}\|P(0,-s)\varphi - m_{-s} \varphi\|_{L^p(\R^d,
\mu_0)} = 0$. To this aim we shall prove  that, for every sequence
$(t_n)\to   \infty$, there exists a subsequence $(s_n)$ such that
$\lim_{n\to  \infty} \|P(0,-s_n)\varphi - m_{-s_n}
\varphi\|_{L^p(\R^d, \mu_0)} = 0$.

Set $f(s, x) :=  \varphi(x)$ for any $(s,x)\in\R^{1+d}$.
Then, $f\in C_b(\U \times \R^d)$, and for every $t>0$,  $s\in \R$ and $ x\in \R^d$ we have
\begin{eqnarray*}
\T (t)(I-\Pi)f (s,x)   = P(s, s-t )\varphi(x) - m_{s-t }\varphi.
\end{eqnarray*}
Formula \eqref{CompAsT(t)} implies
%
%
\begin{eqnarray*}
\lim_{t\to \infty} \int_{-T}^{0} \| P(s, s-t )\varphi - m_{s-t
}\varphi\|_{L^p(\R^d, \mu_s)} ^{p} ds =0.
\end{eqnarray*}
Since $P(0, s-t )\varphi(x) - m_{s-t }\varphi = P(0,s)[  P(s, s-t
)\varphi - m_{s-t }\varphi]$ for $s\in [-T,0]$ and $t\geq 0$, and
\eqref{contraction} holds, then
\begin{equation}
\lim_{t\to \infty}   \int_{-T}^{0}   \| P(0, s-t )\varphi - m_{s-t
}\varphi\|_{L^p(\R^d, \mu_0)} ^{p} ds =0.
\label{limite}
\end{equation}
It follows that  for every sequence $(t_n)\to   \infty$ there exist a subsequence $(s_n)$ and a set $\Gamma \subset [-T,0]$, with negligible complement, such that
\begin{equation}
\label{limite_succ}\lim_{n\to \infty}  \|P(0, s -s_n)\varphi - m_{s-
s_n }\varphi \|_{L^p(\R^d, \mu_0)}  =0,\qquad\;\,s\in\Gamma.
\end{equation}
Our aim is to show that $0\in \Gamma$. This follows from the uniform
continuity in $(-\infty,0]$ of the $C_b(\R^d)$-valued function
$s\mapsto P(0, s )\varphi$ (see Theorem \ref{Th:P(t,s)}(ii)) and of the
real-valued function $s\mapsto m_{s }\varphi$ (see Corollary
\ref{cor:continuity}). Indeed, for each $s\in \Gamma$ we have
\begin{align*}
 \| P(0,  -s_n)\varphi - m_{ - s_n }\varphi \|_{L^p(\R^d, \mu_0)}  &  \leq  \| P(0,-s_n) \varphi - P(0, s -s_n)\varphi   \|_{L^p(\R^d, \mu_0)}
 \\
 & +  \|P(0, s -s_n)\varphi - m_{s- s_n }\varphi \|_{L^p(\R^d, \mu_0)}
 \\
 & + | m_{s- s_n }\varphi - m_{ - s_n }\varphi |,
\end{align*}
and $ \|P(0,-s_n) \varphi -P(0, s -s_n)\varphi   \|_{L^p(\R^d,
\mu_0)}  \leq  \|P(0,-s_n) \varphi - P(0, s -s_n)\varphi
\|_{\infty}$. Then, for every $\varepsilon >0$ there exists $\delta
>0$ such that $ \| P(0,-s_n) \varphi - P(0, s -s_n)\varphi
\|_{L^p(\R^d, \mu_0)} \leq \varepsilon$ and $ | m_{s- s_n }\varphi -
m_{ - s_n }\varphi | \leq \varepsilon$, for each $s\in (-\delta,0)$
and $n\in\N$. Fix $s\in \Gamma \cap (-\delta, 0)$. By
\eqref{limite_succ}, $\| P(0, s -s_n)\varphi - m_{s- s_n }\varphi
\|_{L^p(\R^d, \mu_0)} \leq \varepsilon$ for $n$ large enough, say
$n\geq n(s,\varepsilon)$. Summing up, $ \| P(0,  -s_n)\varphi - m_{
- s_n }\varphi \|_{L^p(\R^d, \mu_0)}\leq 3 \varepsilon$ for $n\geq
n(s,\varepsilon)$. Therefore, $0\in \Gamma$ and
 \eqref{compas_indietro} holds for $\varphi \in C^{\infty}_{c}(\R^d)$.

Let us now fix $\varphi \in C _{b}(\R^d)$, and let $(\varphi_n)$ be a bounded sequence of test functions that converges to $\varphi $
locally uniformly.  By Corollary \ref{cor:continuity}(b), $\lim_{n\to \infty} \sup_{s\in \R}  \|\varphi_n-\varphi\|_{L^p(\R^d,\mu_s)}=0$. Since
\begin{align}
\|P(t,s)\varphi-m_s\varphi\|_{L^p(\R^d,\mu_t)}  \le & \;
\|P(t,s)(\varphi-\varphi_n)\|_{L^p(\R^d,\mu_t)} +
\|P(t,s)\varphi_n-m_s\varphi_n\|_{L^p(\R^d,\mu_t)} \notag
\\
  +  &  \;
|m_s\varphi_n-m_s\varphi| \label{Step2}
\\
  \le & \; 2\sup_{s\in\R}\|\varphi-\varphi_n\|_{L^p(\R^d,\mu_s)} +
\|P(t,s)\varphi_n-m_s\varphi_n\|_{L^p(\R^d,\mu_t)}, \notag
\end{align}
for every $n\in \N$, then
$\|P(t,s)\varphi-m_s\varphi\|_{L^p(\R^d,\mu_t)}$ goes to $0$ as
$s\to -\infty$.

\vspace{2mm}
\noindent {\em Step 3:   (i) implies (iii). }
Let  $\varphi \in C_b(\R^d)$.
Changing variable  in \eqref{limite} we get
\begin{eqnarray*}
\lim_{t\to \infty}   \int_{0}^{T} \int_{\R^d} | P(s +t , s
)\varphi(x) - m_{s}\varphi |^p \mu_{s +t }(dx)\,ds =0,
\end{eqnarray*}
so that there exists  a sequence $(t_n)\to \infty$ such that, for almost every $s \in (0,T)$ and by periodicity for almost every $s\in \R$, we have
\begin{equation}
\label{sottosucc} \lim_{n\to \infty}   \int_{\R^d} | P(s+t _n,
s)\varphi(x) - m_{s}\varphi |^p  \mu_{s +t _n}(dx)  =0.
\end{equation}
Let $\Gamma '$ be the set of all $s\in \R$ such that \eqref{sottosucc} holds. For $ s\in \Gamma ' $ and  for  $t\in [t_n, t_{n+1})$ we have
\begin{eqnarray*}
P(s+t  , s)\varphi - m_{s}\varphi =  P(s+t,s+t_n) [P(s+t _n, s)\varphi - m_{s}\varphi],
\end{eqnarray*}
so that, from \eqref{contraction},
\begin{eqnarray*}
\|P(s+t  , s)\varphi - m_{s}\varphi \|_{L^p(\R^d,  \mu_{s +t  })} \leq
\|P(s+t _n, s)\varphi - m_{s}\varphi \|_{L^p(\R^d,  \mu_{s +t _n})}.
\end{eqnarray*}
Hence, $\lim_{t\to \infty} \| P(t,s)\varphi -
m_s\varphi\|_{L^p(\R^d, \mu_t)} =0$ for $s\in \Gamma '$. To prove
that the limit is zero for every $s\in \R$, we argue as at the end
of Step 2, replacing $ \varphi_n$ by $P(s,r_n)\varphi $, with
$r_n\in \Gamma '$, $r_n\uparrow s$ as $n\to \infty$. Indeed, by
Theorem \ref{Th:P(t,s)}(ii), $P(s,r_n)\varphi $ converges to $\varphi $ locally uniformly.
Estimates \eqref{Step2} imply the statement.

If $\varphi \in L^p(\R^d, \mu_s)$,   \eqref{compas} follows
approaching $\varphi $ by a sequence of functions in $C_b(\R^d)$ and
recalling that $P(t,s)$ and $m_s$ are contractions from $L^p(\R^d,
\mu_s)$ to $L^p(\R^d, \mu_t)$.  So, statement (iii) holds.

 \vspace{2mm}
\noindent {\em Step 4: if Hypothesis  \ref{hyp2} holds, (ii) and
(iii) imply (iv) and (v), respectively.} Let $\varphi \in
C_{b}(\R^d)$. Then, for every $t\in \R$, the functions  $
P(t,s)\varphi - m_s\varphi$ ($s\leq t-1$) are equibounded and
equicontinuous by estimate \eqref{boh}. By the Arzel\`a-Ascoli
Theorem, for every $R>0$ there exist a sequence $(s_n)\to - \infty$
and a function $g\in C_b(B(0,R))$ such that $\lim_{n\to \infty} \|
P(t,s_n)\varphi - m_{s_n}\varphi - g\|_{L^{\infty}(B(0,R))} =0$.

Let $\rho $ be the continuous positive version of the density of
$\mu$ with respect to the Lebesgue measure. Then, $\mu_s =
\rho(s,x)dx$ for any $s\in\R$, by the remark after the proof of
Proposition \ref{Pr:uniqueness}. For each $n\in \N$ we have
\begin{align*}
&\inf_{\R \times B(0,R)} \rho \int_{B(0,R)} |  P(t,s_n)\varphi(x) -
m_{s_n}\varphi  | dx
 \leq    \int_{B(0,R)} |  P(t,s_n)\varphi(x) - m_{s_n}\varphi  | \rho(t,x) dx  \\
&= \int_{B(0,R)} |  P(t,s_n)\varphi(x) - m_{s_n}\varphi  | d\mu_t\
 \leq \bigg (\int_{\R^d} |  P(t,s_n)\varphi(x) - m_{s_n}\varphi  |^p d\mu_t\bigg )^{\frac{1}{p}}.
\end{align*}
If (ii) holds, the last term vanishes as $n\to \infty$.  Therefore,  $\int_{B(0,R)} |  g(x) | dx =0$, so that $g\equiv 0$
and (v) holds.

The proof that (iii) implies (v) is the same.

\vspace{2mm} \noindent {\em Step 5: (ii),  (iii), (iv), (v) imply
(i).} Let $f(s,x) = \alpha (s)\varphi(x)$, with $\alpha \in C(\U)$
and $\varphi \in C_b(\R^d)$. Then $\T(t)(f-\Pi f)(s,x) = \alpha
(s-t) (P(s, s-t)\varphi (x) - m_{s-t}\varphi )$, so that
\begin{equation}
\label{uguaglianza} \| \T(t)(f-\Pi f)\|_{p}^{p} = \frac{1}{T}
\int_0^T |\alpha(s-t)|^p \int_{\R^d} |P(s, s-t)\varphi(x) -
m_{s-t}\varphi |^p \mu_{s}(dx)\, ds.
\end{equation}

If (ii) holds, then
  $\lim_{t\to \infty}  \int_{\R^d} |P(s, s-t)\varphi(x) - m_{s-t}\varphi |^p \mu_{s}(dx) =0$ for each $s\in (0,T)$. Moreover,
$ \int_{\R^d} |P(s, s-t)\varphi(x) - m_{s-t}\varphi |^p
\mu_{s}(dx)\leq (2\|\varphi\|_{\infty})^p$, for each $s\in (0,T)$.
If (iv) holds, $ |\alpha(s-t) (P(s, s-t)\varphi(x) - m_{s-t}\varphi
)|^p $ goes to zero pointwise, and it does not exceed
$(2\|\alpha\|_{\infty} \|\varphi\|_{\infty})^p$, for each $s\in
(0,T)$.  In both cases, by dominated convergence $\lim_{t\to \infty}
\| \T(t)(f-\Pi f)\|_{L^p (\U \times \R^d , \mu)} =0$.

If (iii) or (v)  holds, let us rewrite \eqref{uguaglianza} as
\begin{eqnarray*}
\| \T(t)(f-\Pi f)\|_{L^p (\U \times \R^d , \mu)}^{p} = \frac{1}{T}
\int_0^T |\alpha(s)|^p \int_{\R^d} |P(s +t, s)\varphi(x) -
m_{s}\varphi |^p \mu_{s+t}(dx)\, ds.
\end{eqnarray*}
Then,  $\lim_{t\to \infty}  \int_{\R^d} |P(s+t, s)\varphi(x) -
m_{s}\varphi |^p \mu_{s+t}(dx) =0$ for   each $s\in (0,T)$ if (iii)
or (v) hold. If (iii) holds, this is immediate. If (v) holds, it is
sufficient to use  the uniform convergence of $ |P(s+t, s)\varphi - m_{s}\varphi |^p$ to zero  as $t\to +\infty$, on
each ball $B(0,R)$ and Corollary \ref{cor:continuity}. In both cases, we have again $ \int_{\R^d}
|P(s+t, s)\varphi(x) - m_{s}\varphi |^p \mu_{s+t}(dx)\leq
(2\|\varphi\|_{\infty})^p$, for each $s\in (0,T)$. By dominated
convergence, $\lim_{t\to \infty} \| \T(t)(f-\Pi f)\|_{L^p (\U \times
\R^d , \mu)} =0$.

Since the linear span of the functions $f(s,x) = \alpha(s)\varphi(x)$, with $\alpha \in C(\U)$ and $\varphi \in C_b(\R^d)$, is dense in $L^p (\U \times \R^d , \mu)$, (i) follows.
  \end{proof}

\begin{theorem}
\label{Pr:EquivExp} Let Hypothesis \ref{hyp1} hold. Fix $1\leq p
\leq \infty$, $M>0$, $\omega \in \R$. The following conditions are
equivalent:
 \begin{enumerate}[{\rm (a)}]
 \item
 for every $t>0$ and $u\in  L^p(\U \times \R^{d},\mu)$,
\begin{eqnarray*}
\|\T(t)(I-\Pi)u\|_p \leq Me^{\omega t} \|u\|_p, \quad t>0, \; u\in
L^p(\U \times \R^{d},\mu);
\end{eqnarray*}
 \item
for every $t>s$ and $ \varphi \in L^p(\R^d, \mu_s)$,
\begin{eqnarray*}
\;\;\;\;\;\;\;\;\;\| P(t,s)\varphi  - m_s \varphi\|_{L^p(\R^d, \mu_t)} \leq
Me^{\omega(t-s)}\|\varphi \|_{L^p(\R^d, \mu_s)}, \quad t>s , \;
\varphi \in L^p(\R^d, \mu_s).
\end{eqnarray*}
 \end{enumerate}
\end{theorem}
\begin{proof}
For $p=\infty$ the equivalence  is immediate.

The proof that (a)$\Rightarrow $(b) for $p<\infty$ is quite similar to the proof of Step 2 of  \cite[Thm.~2.17]{GL2},  that concerns $p=2$ and backward Ornstein-Uhlenbeck evolution operators. In our periodic case we do not need the localization function $\xi$ of \cite{GL2}, it is sufficient to define $u(s,\cdot) = \varphi $ for every $s$.  We omit  the details of the proof, leaving them to the reader.

Still for $p<\infty$, (b)$\Rightarrow $(a) is easy. For, if (b) holds, then for $s\in \R$,
$t>0$, and $u\in L^p(\U \times \R^{d},\mu)$, we have
\begin{eqnarray*}
\int_{\R^d} |P(s, s-t)u(s-t,\cdot) - m_{s-t}u(s-t, \cdot)|^p d\mu_s
\leq   M^p e^{\omega pt} \int_{\R^d}|u(s-t, \cdot)|^p d\mu_{s-t},
\end{eqnarray*}
and integrating over $[0,T]$ we obtain
\begin{align*}
\int_{\U \times \R^d}|\T(t)(I-\Pi)u|^p d\mu
&\leq  M^p e^{\omega pt} \frac{1}{T} \int_{0}^{T} \int_{\R^d}|u(s-t, \cdot)|^p d\mu_{s-t}ds\\
&=  M^p e^{\omega pt} \frac{1}{T} \int_{0}^{T} \int_{\R^d}|u(\tau , \cdot)|^p d\mu_{\tau }d\tau\\
&=  M^p e^{\omega pt}\|u\|_p^p.
\end{align*}
 \end{proof}

\begin{remark}
{\rm It is also possible to relate the asymptotic behavior  of
$\nabla_x\T(t)$ to the asymptotic behavior of  $\nabla_xP(t,s) $.
Namely, results similar to  Theorems \ref{Pr:equivalenza} and
\ref{Pr:EquivExp} hold, with $|\nabla_x \T(t)u |$ and
$|\nabla_xP(t,s)\varphi |$ replacing $\T(t) (I-\Pi)u$ and
$P(t,s)\varphi -m_s\varphi$, respectively. The details are left to
the reader.} \end{remark}

In view of Theorems \ref{Pr:equivalenza} and \ref{Pr:EquivExp}, we study the decay to
zero of $\T(t)(I-\Pi)$. The starting point is the decay  of
$|\nabla_x \T(t)f|$ as $t \to \infty$, for every $f\in L^2(\U
\times \R^d, \mu)$. Since everything relies on formula
\eqref{integration<Hu}, we need that the assumptions of Proposition
\ref{lemma1 parte1} hold. The proof of the following proposition is
an extension to the evolution semigroup of a similar proof for
Markov semigroups generated by elliptic operators (e.g., \cite{DPG}).

\begin{proposition}
\label{DecadimentoGradiente} Let   Hypotheses \ref{hyp1} and
\ref{hyp2} hold. If the diffusion coefficients are bounded, or if
\eqref{stimapiu'} is satisfied, then for every $f\in L^2(\U \times
\R^d, \mu)$ we have
\begin{equation}
\label{decad2} \lim_{t\to\infty} \| \, | \nabla_x \T(t)f|\,\|_2 =0.
\end{equation}
\end{proposition}
\begin{proof}   Let $f \in D(G_2)$. From the equality
\begin{eqnarray*}
\frac{d}{dt} \|\T(t)f \|_{2}^{2} = 2\langle T(t)f, G_2\T(t)f\rangle
_{L^2(\U\times \R^d, \mu)},\qquad\;\,t>0,
\end{eqnarray*}
we obtain
\begin{eqnarray*}
\|T(t)f \|_{2}^{2} - \| f \|_{2}^{2} = 2\int_{0}^{t}\int_{\U\times
\R^d}  \T(s)f\, G_2 \T(s)f \, d\mu\; ds,\qquad\;\,t>0,
\end{eqnarray*}
and using \eqref{integration<Hu}, that holds for the functions in
$D(G_2)$ by Corollary \ref{cor:2}(b), we get
\begin{equation}
\label{midnight}
\|\T(t)f \|_{2}^{2} +2\int_{0}^{t}\int _{\U\times \R^d}  \langle
Q\nabla _{x}\T(s)f , \nabla_x\T(s)f \rangle  d\mu\, ds \leq  \| f
\|_{2}^{2},\qquad\;\, t>0.
\end{equation}
Therefore, the function
\begin{eqnarray*}
\chi _{f }(s) := \int _{\U\times \R^d}   | \nabla _{x}
\T(s)f |^{2}  d\mu , \;\; s\geq 0,
\end{eqnarray*}
is in $L^{1}(0,\infty)$, and its $L^{1}$-norm does not exceed $\| f
\|_{2}^{2}/\eta_0$. Its derivative is
\begin{eqnarray*}
\chi _{f}'(s) = \int _{\U\times \R^d}   2  \langle  \nabla_{x}
T(s)f , \nabla_{x}T(s)G_2f \rangle  \, d\mu
\end{eqnarray*}
so that, if $f\in D((G_2)^2)$,
\begin{eqnarray*}
   |\chi _f'(s)|   \leq    2 \bigg(\int _{\U\times \R^d}    |\nabla_{x}
T(s)f |^{2}d\mu\bigg)^{\frac{1}{2}} \bigg(\int _{\U\times \R^d}
|\nabla_{x} T(s)G_2 f|^{2}d\mu\bigg)^{\frac{1}{2}} \leq  \chi
_{f}(s) + \chi _{G_2f}(s),
\end{eqnarray*}
for any $s\ge 0$. Therefore, also $\chi _{f}'$ is in
$L^{1}(0,\infty)$. This implies that $\lim_{s\to\infty }\chi _{f}(s)
=0$, and \eqref{decad2}  holds for every $f \in D((G_2)^2)$. For
general $f \in L^2(\U \times \R^d, \mu)$, \eqref{decad2} follows
approaching  $f$
by a sequence of functions in $D((G_2)^2)$, which is dense in
$L^2(\U \times \R^d, \mu)$, and using Corollary \ref{cor:2}(a).
\end{proof}

\begin{theorem}
\label{decadimento}
Let  the assumptions of Proposition \ref{DecadimentoGradiente} hold. Then,
for every $p\in [1,\infty)$
\begin{equation}
\label{decadimentoL^p}
\lim_{t\to \infty} \|\T(t)(I-\Pi)f\|_p =0, \quad f\in L^p(\U \times \R^d,\mu).
\end{equation}
Therefore, statements (ii) to (iv) of Theorem \ref{Pr:equivalenza} hold.
\end{theorem}
\begin{proof}
Let $f\in \mathcal C$, which is dense in $L^p(\U \times \R^d,\mu)$
by Proposition \ref{Pr:core}. Then, $f$ is a linear combination of
functions $ u_{\tau, \chi, \alpha}$, defined before Proposition \ref{Pr:core}.

Let us prove that, for each $u =  u_{\tau, \chi, \alpha}$, the set of
functions $\{ \T(t)(I-\Pi)u: t>0\}$  is equicontinuous and
equibounded in $\R \times B(0,R)$, for each $R>0$.

Since $\Pi u (s,x) = \alpha(s) m_{\tau}\chi$, then $\T(t)\Pi u (s,x)
= \alpha(s-t)m_{\tau}\chi$ is equicontinuous and equibounded.
Concerning $\T(t)u$, we recall that it is the time periodic
extension of the function $(s,x)\mapsto \alpha(s-t)P(s,\tau)\chi(x)$
defined for $s\in [a +t, a +t +T)$, $x\in \R^d$, if the support of
$\alpha$ is contained in $(a, a+T)$ with $a\geq \tau$. We have to
prove only equicontinuity, since $\|\T(t)
u\|_{\infty}\leq\|\alpha\|_{\infty} \|\chi\|_{\infty}$. By Theorem
\ref{Th:gradP(t,s)infinito}, $\|\,| \nabla_x
\alpha(s-t)P(s,\tau)\chi|\,\|_{L^{\infty}(\R^d)} \leq
C_1\|\alpha\|_{\infty} \|\chi\|_{C^1_b(\R^d)}$, so that $\T(t)u$ is
equi-Lipschitz continuous in $x$. To prove that it is equi-Lipschitz
continuous in $s$ we show preliminarily that, for every $R>0$,
\begin{equation}
\label{eq:MaggLoc} \sup _{s\geq \tau, |x|\leq R}|\A
(s)P(s,\tau)\chi(x)| < \infty.
\end{equation}
From the proof of \cite[Thm. 2.2]{KLL} we know  that the function
$(s,x)\mapsto P(s,\tau)\chi(x)$ belongs to $C^{1+\alpha/2,
2+\alpha}_{\loc}([\tau, \infty)\times \R^d)$ and, therefore,
\begin{eqnarray*}
\sup_{\tau \leq s\leq \tau + 2T, \,|x| \leq R}|\A(s)P(s,\tau)\chi(x)| <\infty.
\end{eqnarray*}
 If $s\in (\tau +kT, \tau + (k+1)T]$
with $k\geq 2$ we write
\begin{eqnarray*}
P(s,\tau)\chi = P(s,\tau +(k-1)T)P(\tau + (k-1)T, \tau)\chi :=
P(\sigma, \tau)\varphi,
\end{eqnarray*}
with $\sigma = s-(k-1)T\in (\tau +T, \tau +2T]$, $\varphi = P(\tau +
(k-1)T, \tau)\chi \in C_b(\R^d)$, $\|\varphi \|_{\infty} \leq
\|\chi\|_{\infty}$. By Theorem \ref{Th:P(t,s)}(i),
\begin{eqnarray*}
\sup \{ |\A(\sigma )P(\sigma , \tau) \varphi : \tau +T\leq \sigma
< \tau + 2T, \;\,|x|\leq R \} \leq C(R)\|\varphi \|_{\infty},
\end{eqnarray*}
and \eqref{eq:MaggLoc} follows.

From the equality
\begin{eqnarray*}
D_s \T(t)u(s, \cdot)  = \alpha'(s-t) P(s,\tau)\chi  + \alpha(s-t) \A
(s)P(s,\tau)\chi , \quad s\in [a +t, a +t +T),
\end{eqnarray*}
using  \eqref{eq:MaggLoc} we obtain that $D_s\T(t)u$ is bounded in
$[a +t, a +t +T) \times B(0,R)$. Since it is periodic in $s$,
it is bounded in $\R \times B(0,R)$.

Therefore, for each $f\in \mathcal C$ the set of functions $\{
\T(t)f: t>0\}$  is equicontinuous and equibounded in $\R \times
B(0,R)$, for each $R>0$. By the Arzel\`a-Ascoli Theorem and the
usual diagonal procedure, there exist a sequence $t_n \to \infty$
and a function $g\in C_b(\U \times \R^d)$ such that
$\T(t_n)(I-\Pi)f$ converges to $g$ uniformly on $\U \times B(0,R)$,
for each $R>0$. Since $\|\T(t_n)(I-\Pi)f\|_{\infty}\leq
\|f\|_{\infty}$, by dominated convergence $\T(t_n)(I-\Pi)f$
converges to $g$ in $L^p(\U \times \R^d,\mu)$, for every $p\in [1,
\infty)$.

Let us prove that $g\equiv 0$. We have $\lim_{n\to\infty}\|\T(t_n)(I-\Pi)f - g\|_{2} =0$,
moreover, by Proposition \ref{DecadimentoGradiente},
$\lim_{n\to\infty}\|\,|\nabla_x \T(t_n)(I-\Pi)f|\, \|_{2}=0$. Since the density $\rho$ of $\mu$ with
respect to the Lebesgue measure is positive, the space derivatives
are closed operators in $L^2(\U \times \R^d,\mu)$. This implies that
$g\in W^{0,1}_2(\U \times \R^d,\mu)$ has null space derivatives, so
that it depends only on $s$. On the other hand, $g\in (I-\Pi)(L^2(\U
\times \R^d,\mu))$ because it is the limit of the sequence
$(\T(t_n)(I-\Pi)f)$ that has values in $(I-\Pi)(L^2(\U \times \R^d,\mu))$.
If a function in $(I-\Pi)(L^2(\U \times \R^d,\mu))$ is independent
of the space variables, it vanishes. Therefore, $g\equiv 0$. Since
the only possible limit $g$ is zero, then $\lim_{t\to \infty}
\|\T(t )(I-\Pi)f  \|_{p} =0$, for every $p\in [1, \infty)$.

Since $\mathcal C$ is dense in $L^p(\U \times \R^d,\mu)$ and
$\|\T(t)(I-\Pi)\|_{{\mathcal L}(L^p(\U \times \R^d,\mu))} \leq 1$,
\eqref{decadimentoL^p}  follows. Theorem \ref{Pr:equivalenza}  yields the other statements.
\end{proof}

For   $\varphi\in C_b(\R^d)$, the convergence of $P(t,s)\varphi -
m_s\varphi$ to $0$ is not uniform in $\R^d$ in general, even in the
autonomous case. Take for instance any Ornstein-Uhlenbeck operator
${\mathcal A}$,
\begin{eqnarray*}
{\mathcal A}\varphi =\frac{1}{2}\tr  (QD^2\varphi) + \langle Bx, \nabla \varphi \rangle,
\end{eqnarray*}
where $Q $ is  symmetric and positive definite   and all the
eigenvalues of  $B $ have negative real part. Then, the
Ornstein-Uhlenbeck semigroup $T(t)$ has a unique invariant measure
$\mu$,  which is the Gaussian measure with zero mean and covariance
operator $Q_{\infty} :=  \int_0^{\infty} e^{sB}Qe^{sB^*}ds$. We have
$P(t,s) = T(t-s)$ and $\mu_t = \mu$ for every $t\in \R$.

Take an exponential function $g=e^{i\langle\cdot,h\rangle}$ ($h\in\R^d$). Then
\begin{eqnarray*}
T(t)g = \exp\left (-\textstyle{\frac{1}{2}}\langle Q_t h,h\rangle
+i\langle \cdot,e^{tB^*}h\rangle\right ),\qquad\;\,t\ge 0,
\end{eqnarray*}
where $Q_{t} :=  \int_0^{t}
e^{sB}Qe^{sB^*}ds$. A simple computation shows that
$\int_{\R^d}g\, d\mu = e^{-\langle Q_{\infty} h,h\rangle /2}$.
Therefore,
\begin{eqnarray*}
\;\;\;\;\;\;\; T(t)g - \int_{\R^d}g \, d\mu &\hs{5}=\hs{5}&
\left \{ \exp\left (-\textstyle{\frac{1}{2}}\langle Q_{t}  h,h\rangle\right )
- \exp\left (-\textstyle{\frac{1}{2}}\langle Q_{\infty}  h,h\rangle\right )\right\}e^{i\langle \cdot,e^{tB^*}h\rangle}\\
&\hs{5}&+ \exp\left (-\textstyle{\frac{1}{2}}\langle
Q_{\infty} h,h\rangle\right ) \left (\exp( i\langle
\cdot,e^{tB^*}h\rangle)-1\right ),
\end{eqnarray*}
for any $t>0$. The sup norm of the  first addendum in the right-hand side
 vanishes  as $t\to\infty$  but the second one does not, since, for any
$t>0$, $\sup_{x\in\R^d}|\exp( i\langle x,e^{tB^*}h\rangle)-1| =\sup
_{\theta\in\R}| \exp( i\theta )-1| $ $=2$.

\vspace{2mm}

Concerning  exponential rates of convergence, for every $p\in [1, \infty)$ let us define  the right half-lines
\begin{align*}
A_p : = \{ & \omega\in\R: \exists M_{\omega}>0 \;\;\mbox{s.t.} \;\,
\|\T(t)(f-\Pi f)\|_p\le M_{\omega} e^{\omega t}\|f-\Pi f\|_p
\;\;\mbox{for any}\;\,t\geq 0,\notag
\\
&f\in L^p(\U \times \R^{d},\mu)\},\\[1mm]
B_p := \{ & \omega\in\R: \exists
N_{\omega}>0 \;\;\mbox{s.t.}\;\,\|\,|\nabla_x\T(t)f|\, \|_p\le
N_{\omega} e^{\omega t}\|f\|_p\;\;\mbox{for any}\;\,t\geq 1, \notag
\\
& f\in L^p(\U \times \R^{d},\mu)\} ,
\end{align*}
and their infima
\begin{equation}
\label{omegap,gammap}
 \omega_{p}: =\inf A_p, \quad  \gamma_{p} :=\inf B_p.
\end{equation}
Then, $\omega_p \leq 0$ is  the growth bound of the part of $\T(t)$
in $(I-\Pi)( L^p(\U \times \R^{d},\mu))$. We recall that if   $\ell_p <\infty$, then $\ell_p\in B_p$
by Proposition \ref{thm-2.7},  hence $\gamma_p\leq \min\{\ell_p,0\}$.

\begin{theorem}
\label{gamma=omega}
Let  Hypotheses  \ref{hyp1}  and  \ref{hyp2} hold. Then  $A_p\subset B_p$ for every $p\in (1, \infty)$ such that
$\ell_p<\infty$. If the diffusion coefficients are bounded, $B_p \subset A_p $ for every $p\geq 2$.
\end{theorem}
\begin{proof}
Let
$\ell_p<\infty$. By Proposition \ref{thm-2.7}, $\T(t)$ maps $L^p(\U
\times \R^d, \mu)$ into $W^{0,1}_{p}(\U \times \R^d, \mu)$ for every
$t>0$.

 Fix $f\in L^p(\U \times \R^d, \mu)$ and $\omega \in A_p$.
Since $\Pi f$ is independent of $x$,
$\nabla_x\T(t)f=\nabla_x\T(t)(f-\Pi f)$. Taking
\eqref{grad-p-semigruppone>1} into account, for $t>1$ we   estimate
\begin{align*}
\|\nabla_x\T(t)(f-\Pi f)\|_p&=\|\nabla_x\T(1)(\T(t-1)(f-\Pi f))\|_p\\
&\le C_3e^{\ell_p}\|\T(t-1)(f-\Pi f)\|_p\\
&\le C_3e^{\ell_p}M_{\omega} e^{\omega(t-1)}\|f-\Pi f\|_p,
\end{align*}
so that $  \omega \in B_p$, and the first part of the statement is proved.

\vspace{2mm}

If the diffusion coefficients are bounded, set
\begin{equation}
\label{Lambda}
\Lambda:= \sup\{ \langle Q(s,x)\xi, \xi\rangle: s\in \U,\; x\in \R^d,  \;\xi\in
\R^d, \,|\xi|=1\}.
\end{equation}
Since $A_p\supset [ 0,\infty)$, if  $B_p\subset [0,\infty)$ the inclusion $B_p \subset A_p$ is obvious. So, we may assume that
$B_p \cap (-\infty , 0) \neq \varnothing$.

Fix $f\in (I-\Pi)(D(G_p))$ and $\omega \in B_p$,
$\omega <0$.  Then,
\begin{eqnarray*}
\frac{d}{dt}  \int_{\U \times \R^d} |\T(t)f|^pd\mu =  p \int_{\U
\times \R^d} |\T(t)f|^{p-2}\T(t)f  \,G_p\T(t )f d\mu ,
\end{eqnarray*}
so that, by \eqref{integration<Hu} and the H\"older inequality,
\begin{align*}
\frac{d}{dt}  \int_{\U \times \R^d} |\T(t)f|^pd\mu &=  -p(p-1)
\int_{\U \times \R^d}|\T(t)f|^{p-2}\langle Q\nabla_x \T(t)f, \nabla_x\T(t)f\rangle d\mu\\
&\ge -p(p-1)\Lambda \|\T(t)f\|_{p}^{p-2} \|\,| \nabla_x\T(t)f|\, \|_{p}^{2}\\
&\ge  -p(p-1)\Lambda \|\T(t)f\|_{p}^{p-2} N^2_{\omega} e^{2\omega
t}\|f\|_{p}^{2}.
\end{align*}
Therefore, the function
\begin{eqnarray*}
\beta(t) :=  \|\T(t)f\|_{p}^{2} =  \bigg( \int_{\U \times \R^d}
|\T(t)f|^pd\mu \bigg)^{\frac{2}{p}}, \quad t\geq 1
\end{eqnarray*}
either vanishes in a halfline, or it is strictly positive in $[1, \infty)$, and in this case
\begin{eqnarray*}
\beta'(t) = \frac{2}{p}  \|\T(t)f\|_{p}^{2-p} \frac{d}{dt} \int_{\U
\times \R^d} |\T(t)f|^pd\mu    \geq -2(p-1)\Lambda N^2_{\omega}
e^{2\omega t}\|f\|_{p}^{2}.
\end{eqnarray*}
Since $\lim_{t\to\infty}\beta(t)=0$ by Theorem \ref{decadimento},
then
\begin{eqnarray*}
\beta(t) = -\int_{t}^{\infty} \beta'(s)ds  \leq 2(p-1)\Lambda N^2_{\omega} \|f\|_{p}^{2}
\int_{t}^{\infty} e^{2\omega s}ds = \frac{(p-1)\Lambda
N^2_{\omega}}{|\omega|}  e^{2\omega t}\|f\|_{p}^{2},
\end{eqnarray*}
for any $t\ge 1$, that is,
\begin{eqnarray*}
 \|\T(t)f\|_{p}^{2}  \leq \widetilde{M}_p e^{2\omega t}\|f\|_{p}^{2}, \quad t\geq
 1.
 \end{eqnarray*}
Since $(I-\Pi)(D(G_p))$ is dense in $(I-\Pi)(L^p(\U\times\R^d,\mu))$, the above estimate holds
for any $f\in L^p(\U\times\R^d,\mu)$, and this implies that $ \omega \in A_p$. It
follows that $B_p\subset A_p$.
\end{proof}

The second part of the  proof of Theorem  \ref{gamma=omega} may be easily adapted to the case of unbounded diffusion coefficients, and it yields, for $p\geq 2$,
\begin{eqnarray*}
\|\T(t)f\|_{p}^{2}  \leq C e^{2\omega t}\|f\|_{p}^{2}, \quad f\in L^p(\U \times \R^d, \mu), \;t\geq 1,
\end{eqnarray*}
for every $\omega <0$ such that
\begin{equation}
\label{eq:speriamo}
\exists M:\;\;\| \langle Q\nabla_x\T(t)f, \nabla_x \T(t)f \rangle   \|_{p}  \le M  e^{\omega t}\|f\|_{p}, \quad
f\in L^p(\U \times \R^d, \mu), \; t\geq 1.
\end{equation}
But at the moment we are not able to give any sufficient conditions for \eqref{eq:speriamo} to hold, while a sufficient condition for $\gamma_p <0$ is $\ell_2 <0$, by  estimate \eqref{grad-p-semigruppone>1}.

Theorem  \ref{gamma=omega} has two important consequences. The first
one is about the spectral gap of $G_p$ and the solvability of the
equation $\lambda u - G_p u =f$; the second one is about the
asymptotic behavior of the evolution operator $P(t,s)$.

\begin{corollary}
\label{Cor:spectralgap}
Let Hypotheses  \ref{hyp1} and  \ref{hyp2} hold. Assume that the diffusion coefficients are bounded and that $\ell_2<0$.
Then:
\begin{itemize}
\item[(a)]
 $\sigma (G_p)\cap i\R = \{ 2\pi ik/T: k\in \Z\}$ consists of simple isolated eigenvalues for every $p\in (1, \infty)$.
In particular,  for every $f\in L^p(\U \times \R^d, \mu)$ the parabolic
problem $G_p u
=f$ is solvable  if and only if $\int_{\U \times \R^d}
f\,d\mu =0$. In this case, it has infinite solutions, and the
difference of two solutions is constant.
\item[(b)] For every $p\in (1, \infty)$ $G_p$ has a spectral gap. Specifically,
\begin{eqnarray*}
\;\;\;\;\sup \{ \mbox{\rm Re}\,\lambda: \lambda \in \sigma (G_p)\setminus  i\R \}
\le \left\{ \begin{array}{ll}
\ell_2, & \mbox{\rm if}\; p\geq 2,
\\[2mm]
2\ell_2 (1-1/p), & \mbox{\rm if}\; 1<p< 2.
\end{array}\right.
\end{eqnarray*}
\end{itemize}
\end{corollary}
\begin{proof}
By Proposition \ref{thm-2.7},  $\gamma_2 \leq \ell_2$.
Since $\|\T(t)(I-\Pi)f\|_{1} \leq 2\|f\|_{1}$ for every $t>0$ and $f\in L^1(\U \times \R^d, \mu)$, using estimate \eqref{grad-p-semigruppone>1} with $p=2$ and interpolating between $L^1$ and $L^2$ we get
\begin{eqnarray*}
\|\T(t)(I-\Pi)f\|_{p} \leq M_{ p}e^{2\ell_2 (1-1/p)t} \|f\|_{p},
\end{eqnarray*}
for every $f\in L^p(\U \times \R^d, \mu)$.
Therefore, the spectrum of the part of $G_p$ in $(I-\Pi)(L^p(\U
\times \R^d, \mu))$ is contained in the halfplane Re$\,\lambda \leq
\ell_2$, if $p\geq 2$, and in the halfplane
Re$\,\lambda \leq 2\ell_2 (1-1/p)$, if $1<p<2$.
For the other values
of $\lambda$, it is convenient to write the equation $\lambda u -
G_pu = f$  as the system
\begin{eqnarray*}
\left\{
\begin{array}{l}
\lambda \Pi u - G_{p}\Pi u =\Pi f,
\\[3mm]
\lambda (I-\Pi) u - G_{p}(I-\Pi)u = (I-\Pi )f,
\end{array}\right.
\end{eqnarray*}
where the second equation is uniquely solvable. Setting $\Pi u =
\beta$, the first equation may be rewritten as
\begin{eqnarray*}
\beta \in W^{1,p}(\U,\textstyle\frac{ds}{T}), \quad \lambda \beta(s)
+ \beta ' (s) = m_s f(s, \cdot), \;\, s\in \U ,
\end{eqnarray*}
and it is uniquely solvable   if and only if $\lambda \neq 2\pi
ik/T$ for every $k\in \Z$. Since the eigenvalues $2\pi ik/T$ of the
realization of the first order derivative in $L^p(\U ,
\frac{ds}{T})$ are  simple, the eigenvalues
$2\pi ik/T$ of $G_p$ are simple too. In particular, for $\lambda =0$
the above equation is solvable if and only if $\int_{0}^{T} m_s f(s,
\cdot)ds =0$, which means $\int_{\U \times \R^d} f\,d\mu =0$, and in
this case the solutions differ by constants. The statements follow.
\end{proof}

\begin{corollary}
\label{Cor:ca} Let Hypotheses  \ref{hyp1} and  \ref{hyp2} hold. Assume that the diffusion coefficients are bounded and
that $\ell_2<0$. Then, for every $p>1$   there exists $M_p>0$ such that
\begin{equation}
\label{eq2} \| P(t,s)\varphi  - m_s \varphi \|_{L^p(\R^d, \mu_t)}
\leq M_pe^{\ell_2 (t-s)}\|\varphi \|_{L^p(\R^d, \mu_s)}, \quad t>s,
\; \varphi\in L^p(\R^d, \mu_s),
\end{equation}
if $p\geq 2$,  and
\begin{equation}
\label{eqp}\| P(t,s)\varphi  - m_s \varphi \|_{L^p(\R^d, \mu_t)}
\leq M_pe^{\theta_p(t-s) }\|\varphi \|_{L^p(\R^d, \mu_s)}, \quad
t>s, \; \varphi\in L^p(\R^d, \mu_s),
\end{equation}
if $1<p<2$, with $\theta_p = 2\ell_2(1-1/p)$.
\end{corollary}
\begin{proof}
By estimate \eqref{grad-p-semigruppone>1}, $\ell_2\in B_p$ for
$p\geq 2$, and  \eqref{eq2} follows applying Theorem
\ref{Pr:EquivExp}  and Theorem \ref{gamma=omega}. For $1<p<2$,
the estimate $\|P(t,s)\varphi  - m_s \varphi \|_{L^p(\R^d, \mu_t)} \leq
M_pe^{2\ell_2(1-1/p)(t-s) }\|\varphi \|_{L^p(\R^d, \mu_s)}$
 follows  interpolating between $L^1$ and $L^2$, since $ \| P(t,s)\varphi  - m_s \varphi \|_{L^1(\R^d, \mu_t)} \leq 2 \|\varphi \|_{L^1(\R^d, \mu_s)}$ for
$t>s$ and $\varphi \in L^1(\R^d, \mu_s)$. \end{proof}

To get a better decay estimate in $L^p$ spaces with $p<2$ we need
more refined arguments. An important tool is a logarithmic Sobolev
estimate, that will be proved in the next subsection.

\vspace{2mm}

We end this subsection with a remark. Spectral gaps of elliptic
differential operators with unbounded coefficients and asymptotic
behavior of the associated semigroups are usually proved through
Poincar\'e inequalities. We may
prove a Poincar\'e type inequality in our nonautonomous
setting, and precisely
\begin{proposition}
Let  Hypotheses \ref{hyp1} and
\ref{hyp2} hold. Assume that  the diffusion coefficients are bounded and that $\ell_2<0$. Then
\begin{equation}
\label{Poincare'}
\int_{\U \times \R^{d}} |f-\Pi f|^2 d \mu \leq \frac{
\Lambda}{|\ell_2|}\int_{\U \times \R^{d}}|\nabla_x f|^2 d\mu ,\quad
f\in W^{0,1}_{2}(\U \times \R^{d}, \mu),
\end{equation}
where $\Lambda$ is defined in \eqref{Lambda}.
\end{proposition}
\begin{proof} Let $f\in (I-\Pi)(D(G_2))$. By Corollary \ref{cor:2}(c), inequality \eqref{midnight} is in fact an equality. Letting $t\to \infty$ in \eqref{midnight} and recalling that  $\lim_{t\to \infty}\T(t)f =0$ by Theorem \ref{decadimento}, we obtain
$$  \| f \|_{2}^{2}  =  2 \int_{0}^{\infty} \int _{\U\times \R^d}  \langle
Q\nabla _{x}\T(s)f , \nabla_x\T(s)f \rangle  d\mu\, ds $$
and therefore, using \eqref{grad-grad},
\begin{align*}
  \| f \|_{2}^{2} & \leq  2\Lambda \int_{0}^{\infty} \int _{\U\times \R^d}
|\nabla _{x}\T(s)f |^2  d\mu\, ds
 \leq  2\Lambda \int_{0}^{\infty} \int _{\U\times \R^d}  e^{2\ell_2 s} |\nabla _{x} f |^2  d\mu\, ds
 \\
 &
=  \frac{ \Lambda}{|\ell_2|}\int_{\U \times \R^{d}}|\nabla_x f|^2
d\mu,
\end{align*}
so that \eqref{Poincare'} holds for every $f\in D(G_2)$. Since
$ D(G_2)$ is dense in $ W^{0,1}_{2}(\U \times \R^{d}, \mu)$, \eqref{Poincare'} holds for every $f\in  W^{0,1}_{2}(\U \times \R^{d}, \mu)$.
\end{proof}

Once the Poincar\'e inequality \eqref{Poincare'} is established,
arguing as in \cite[Prop.~6.4]{DPL} we   obtain
\begin{eqnarray*}
\| \T(t)(f- \Pi f)\|_2 \leq e^{ \eta_0\ell_2 t/\Lambda} \| f -\Pi
f\|_2, \quad t>0,
\end{eqnarray*}
so that  $\omega_2\leq  \eta_0\ell_2 /\Lambda$. Since $\eta_0 \leq
\Lambda$, the estimate  $\omega_2 \leq \ell_2$ obtained through
Theorem \ref{gamma=omega} is sharper. Such estimates coincide only
if $\eta_0 = \Lambda$, that is if the diffusion matrix $Q$ is a
scalar multiple of the identity.

\subsection{A log-Sobolev type inequality}

Throughout the whole subsection we assume that Hypotheses \ref{hyp1}
and  \ref{hyp2} hold,  that $r_0<0$, and that the diffusion coefficients
 are independent of $x$.
 We recall that $r_0=\sup_{(t,x)\in\U\times\R^d}r(t,x)$ where $r$ is the function in
Hypothesis \ref{hyp2}(ii). This is an important restriction, due to the fact that in the proof (which is an adaptation to the nonautonomous case of the method of \cite[Thm.~6.2.42]{DS}) we use the estimate
  \begin{equation}
|\nabla_x \T(t)f(s,x)| \leq e^{r_0(t-s)}\T(t)|\nabla f|(x),
\qquad\;\, t>0,\;\,(s,x)\in\R^{1+d},
 \label{grad-punt-LS}
\end{equation}
obtained from Theorem \ref{Th:gradP(t,s)puntuale}(iii), which is not
obvious (and, in general, not true) if the diffusion coefficients
are not independent of $x$. We refer the reader to \cite{Wang} for a
discussion about the validity of an estimate similar to
\eqref{grad-punt-LS} in the autonomous case.

\begin{lemma}
\label{conv T media}
For any $f \in D(G_{\infty})$ such that $f \geq \delta$ for some $\delta
>0$,
we have
\begin{eqnarray*}
\lim_{t \to\infty} \int_{\U \times \R^{d }}\T (t)f
\log (\T (t)f)d\mu = \frac{1}{T} \int_0^T \Pi f \log (\Pi f)ds.
\end{eqnarray*}
\end{lemma}
\begin{proof} By the   definition of $\T (t)$ we have
\begin{align*}
\frac{1}{T} \int_0^T \Pi f \log (\Pi f) ds & = \int_{\U \times \R^d } \Pi f(\cdot-t)\log (\Pi f(\cdot-t)) d \mu\\
& = \int_{\U \times \R^{d}} \T (t) \Pi f \log (\T (t)\Pi f) d \mu.
\end{align*}
On the other hand, using H\"older inequality and recalling that the
function $y \mapsto y\log (y)$ is H\"older continuous on bounded
sets, we can determine $C>0$ and $\alpha \in (0,1)$ such that
\begin{align*}
&\left| \int_{\U \times \R^d } \left(\T (t)f \log (\T (t)f)  -
\T (t) \Pi f \log(\T (t) \Pi f) \right) d \mu  \right|\\
&\leq
C \int_{\U \times \R^d } |\T (t)(f- \Pi f) |^{\alpha} d \mu
\leq   C  \|\T (t)(f -  \Pi f)\|_2^{\alpha},
\end{align*}
for any $t>0$, and Theorem \ref{decadimento} yields the assertion.
\end{proof}

We recall that $\Lambda$ is the supremum of the eigenvalues of the matrices $Q(s)$.

\begin{theorem}
\label{logsob thm dcomp delta} For any $p \in [1,\infty)$ and any
$f \in D(G_{\infty})$ with positive infimum we
have
 \begin{equation}
  \label{logsob ineq}
\int_{\U \times \R^d } f^p \log (f^p) d \mu  \leq
\frac{1}{T}\int_0^T \Pi f^p \log (\Pi f^p)ds + \frac{p^2 \Lambda}{2
|r_0|} \int_{\U \times \R^d } f^{p-2} |\nabla_x f|^2 d \mu.
\end{equation}
\end{theorem}
\begin{proof} Let $f \in D(G_{\infty})$ satisfy $f\geq \delta$ for some $\delta >0$.
We first prove \eqref{logsob ineq} with $p=1$. By Proposition \ref{prop-2.9}, $\T (t)f \in D(G_{\infty})$ for any $t >0$. Moreover,
$\T (t)f \geq\T (t) \delta \equiv \delta$ for any $t \geq 0$.

Let us consider  the function $F: [0,\infty) \to \R$ defined by
\begin{eqnarray*}
F(t) =\int_{\U \times \R^d }\T (t)f \log (\T (t)f) d \mu , \qquad
t\ge 0.
\end{eqnarray*}
By Lemma \ref{conv T media} we have
\begin{eqnarray*}
\lim_{t \to\infty} F(t) = \frac{1}{T} \int_0^T \Pi f \log (\Pi f) ds.
\end{eqnarray*}
We want to show that $F$ is differentiable, and to compute $F'(t)$.
First of all we remark that, since $\T(t)f\in D(G_{\infty})$ and
$\T(t)f\geq \delta$, the function $\log (\T(t)f)$ is in
$D(G_{\infty})$ for any $t\ge 0$. Indeed, it belongs to $C_b(\U
\times \R^d)$ $\cap$ $W^{1,2}_{q}(\U \times B(0,R))$  for every $q$
and $R$, and
\begin{eqnarray*}
\G ( \log (\T(t)f) ) = \frac{1}{\T(t)f} \G \T(t)f - \frac{1}{(\T(t)f)^2} \langle Q \nabla_x
\T(t)f, \nabla_x \T(t)f \rangle
\end{eqnarray*}
 is continuous and bounded. Taking Proposition \ref{thm-2.7} into account,
 it follows that $\T(t)f\log(\T(t)f) \in  D(G_{\infty})$
 and
\begin{align}
 \label{two}
  \G [\T(t)f\log(\T(t)f)] =& \; \T(t) f \bigg( \frac{1}{\T(t)f} \G \T(t)f - \frac{1}{(\T(t)f)^2} \langle Q \nabla_x
\T(t)f, \nabla_x \T(t)f \rangle\bigg)\notag\\
&+(\G \T(t) f)\log(\T(t) f)    + 2 \left\langle Q \nabla_x \T(t) f,
\frac{\nabla_x \T(t) f}{\T(t) f}\right \rangle\notag\\
 =& \; {\mathcal G}\T(t)f\hskip -1pt+\hskip -1pt\frac{1}{(\T(t)f)^2}\langle
 Q\nabla_x\T(t)f,\nabla_x\T(t)f\rangle\hskip -1pt +\hskip -1pt(\G\T(t)f)\log(\T(t)f).
\end{align}
A straightforward computation shows that
\begin{eqnarray*}
\frac{d}{dt}\left (\T (t)f\log (\T (t)f)\right )= \G \T (t)f \log(\T
(t)f) + \G \T (t) f ,\qquad\;\,t\ge 0.
\end{eqnarray*}
Using \eqref{two} we get
\begin{eqnarray*}
 \frac{d}{dt}\left (\T (t)f\log (\T (t)f)\right ) = \G [\T (t)f \log (\T (t)f) ]
- \frac{1}{\T (t)f} \langle Q \nabla_x\T (t)f, \nabla_x\T (t)f \rangle ,
\end{eqnarray*}
which is continuous and bounded. Therefore, $F$ is differentiable
and, since  the integral of $ \G [\T (t)f \log (\T (t)f)]$
vanishes, we have
\begin{eqnarray*}
 F'(t) = - \int_{\U \times \R^d } \frac{1}{\T (t)f} \langle Q \nabla_x\T (t)f, \nabla_x\T
(t)f \rangle d \mu  , \qquad t \geq 0.
\end{eqnarray*}
Let us estimate $ F'(t)$. The pointwise estimate \eqref{grad-punt-LS} implies that
\begin{eqnarray*}
|\nabla_x\T(t)f(s,x)|^2 \leq e^{2r_0 t}(\T(t)|\nabla_x
f|(s,x))^2,\qquad\;\,t>0,\;\,(s,x)\in\R^{1+d},
\end{eqnarray*}
so that
\begin{eqnarray*}
 \int_{\U \times \R^d } \frac{1}{\T (t)f}
\langle Q \nabla_x\T (t)f, \nabla_x\T (t)f \rangle d \mu \leq  e^{2r_0 t}  \int_{\U \times \R^d } \frac{\Lambda}{\T (t)f} \left(\T(t) |\nabla_x f|\right)^2
d \mu , \qquad t\ge 0.
\end{eqnarray*}
Moreover, using the H\"older inequality in the representation
formula
\begin{eqnarray*}
\T (t)f(s,x) =\int_{\R^d}
f(s-t,y)p_{s,s-t,x}(dy),\qquad\;\,t>0,\;\,(s,x)\in \U \times \R^d,
\end{eqnarray*}
(see Theorem \ref{Th:P(t,s)}), we get
\begin{eqnarray*}
(\T (t)|\nabla_x f|)^2 = \left(\T (t) \left( \sqrt{f}
\frac{|\nabla_x f|}{\sqrt{f}} \right) \right)^2 \leq (\T (t)f)
\left(\T (t) \left( \frac{|\nabla_x f|^2}{f} \right) \right).
\end{eqnarray*}
Therefore, for each $t\ge 0$ we have
\begin{align*}
 \int_{\U \times \R^d } \frac{1}{\T (t)f}
\langle Q \nabla_x\T (t)f, \nabla_x\T (t)f \rangle d \mu
\leq & \; \Lambda  e^{2r_0 t}  \int_{\U \times \R^d }\T (t) \left(
\frac{|\nabla_x f|^2}{f} \right) d \mu\\[1mm]
= & \; \Lambda  e^{2r_0 t}  \int_{\U \times
\R^d } \frac{|\nabla_x f |^2}{f} d \mu,
\end{align*}
that is
\begin{eqnarray*}
F'(t) \geq  - \Lambda  e^{2r_0 t}  \int_{\U \times
\R^d } \frac{|\nabla_x f |^2}{f} d \mu,\qquad\;\,t\ge 0.
\end{eqnarray*}
Integrating with respect to $t$ in $(0,\infty)$ we get
\begin{eqnarray*}
  \frac{1}{T} \int_0^T \Pi f \log (\Pi f) ds - F(0) =   \int_0^{\infty} F'(t) dt
\geq - \frac{\Lambda}{2 |r_0|} \int_{\U \times \R^d }
\frac{|\nabla_x f|^2}{f} d \mu ,
\end{eqnarray*}
that is formula \eqref{logsob ineq} with $p=1$.

Let now fix $p \in (1,\infty)$. We have
\begin{eqnarray*}
\G ( f^p) = p f^{p-1}\G f + p(p-1)f^{p-2} \langle Q \nabla_x f,
\nabla_x f \rangle,
\end{eqnarray*}
where $f^p \geq \delta^p
>0$, and then, again by Proposition \ref{thm-2.7}, $f^p \in
D(G_{\infty})$. The first part of the proof applied to the function $f^p$
yields the conclusion.
\end{proof}

\begin{proposition}
 \label{dsharp any}
For every $p\in (1,\infty)$ and for every $u \in D(G_{\infty})$ we
have
\begin{equation}
 \label{logsob ineq mod}
\int_{\U \times \R^d } |u|^p \log (|u|^p) d \mu  \leq
\frac{1}{T}\int_0^T \Pi |u|^p \log (\Pi |u|^p)ds + \frac{p^2
\Lambda}{2 |r_0|} \int_{\U \times \R^d } |u|^{p-2} |\nabla_x u|^2 d
\mu .
\end{equation}
In addition, if $u \in D(G_{\infty})$ satisfies
\begin{eqnarray*}
\int_{\U \times \R^d } \frac{|\nabla_x u|^2}{|u|} d \mu  < \infty,
\end{eqnarray*}
then \eqref{logsob ineq mod} holds also for $p=1$.
\end{proposition}
\begin{proof} Fix $u \in D(G_{\infty})$ and define the sequence
\begin{eqnarray*}
u_n := \sqrt{u^2 + \frac{1}{n}},\qquad\;\,n\in\N.
\end{eqnarray*}
A straightforward computation shows that
\begin{eqnarray*}
  \G u_n = \frac{u}{u_n} \G u + \frac{1}{n} \frac{\langle Q \nabla_x u, \nabla_x u \rangle}{f_n^3},
  \end{eqnarray*}
so that $u_n \in D(G_{\infty})$ and, moreover, $u_n \geq
\frac{1}{\sqrt{n}}$ for any $n \in \N$. Therefore, by Theorem
\ref{logsob thm dcomp delta}, we have
\begin{equation}
\label{approx n} \int_{\U \times \R^d } u_n^p \log(u_n^p)d\mu \leq
\frac{1}{T} \int_0^T \Pi u_n^p \log(\Pi u_n^p)ds + \frac{p^2
\Lambda}{2 |r_0|} \int_{\U \times \R^d }u_n^{p-2}|\nabla_x u_n|^2 d
\mu,
\end{equation}
for any $p \in [1,\infty)$ and for any $n \in \N$.

Since $0 < u_n^p \leq \|(u^2+1)^{p/2}\|_{\infty}$ for any
$n\in\N$ and the function $x\mapsto x\log x$ is
 continuous in $[0,\infty)$, the left-hand side of \eqref{approx n} converges to
 $ \int_{\U \times \R^d } |u|^p \log(|u|^p)d\mu$. Similarly, since
 $\Pi u_n^p \leq \|(u^2+1)^{p/2}\|_{\infty}$,
 by the dominated convergence theorem, $\Pi |u|^p
\log(\Pi |u|^p)\in L^1((0,T), ds)$, and
\begin{eqnarray*}
\lim_{n\to\infty}  \frac{1}{T} \int_0^T \Pi u_n^p \log(\Pi
u_n^p)ds =  \frac{1}{T} \int_0^T \Pi |u|^p \log(\Pi |u|^p)ds.
\end{eqnarray*}
Concerning the second integral in the right-hand side of
\eqref{approx n}, if $p\in [1,2)$ we have
\begin{eqnarray*}
0 < u_n^{p-2}|\nabla_x u_n|^2 \leq |u|^{p-2}|\nabla_x
u|^2\chi_{\{u\neq 0\}},\qquad\;\,\mbox{a.e. in } \U\times\R^d,
\end{eqnarray*}
and the right-hand side is in $L^1(\U \times \R^d, \mu)$ by
Proposition \ref{lemma1 parte1}(c) for $p>1$ and by assumption for
$p=1$; if $p\geq 2$ we have
\begin{eqnarray*}
0 < u_n^{p-2}|\nabla_x u_n|^2 \leq  u_{1}^{p-2}|\nabla_x u|^2 ,
\end{eqnarray*}
which is bounded by Proposition \ref{thm-2.7}. In any case, the dominated convergence theorem yields
\begin{eqnarray*}
\lim_{n\to\infty}  \int_{\U \times \R^d }u_n^{p-2}|\nabla_x u_n|^2
d \mu  =  \int_{\U \times \R^d }|u|^{p-2}|\nabla_x u|^2 d \mu
\end{eqnarray*}
and the statement follows.
\end{proof}

Proposition \ref{dsharp any} with $p=2$ would be enough to prove
next compactness Theorem \ref{Th:comp1}.  However, it is interesting
to extend logarithmic Sobolev inequalities as far as possible. To
extend estimate \eqref{logsob ineq mod} to all functions $u\in
D(G_p)$ for $p\geq 2$, we use the following lemma.

\begin{lemma}
\label{Le:nonlineare} For $p\geq 2$ and $u\in D(G_p)$, $|u|^p\in
D(G_1)$, and the mapping $u\mapsto |u|^p$ is continuous from
$D(G_p)$ into $D(G_1)$. Moreover, there exists $C_p>0$ such
that $\|
\,|u|^p\,\|_{D(G_1)} \leq C_p\|u\|_{D(G_p)}^p$, for every $u\in
D(G_p)$.
\end{lemma}
\begin{proof}
In the proof of Proposition \ref{lemma1 parte1} we have shown that $|u|^p\in D(G_{\infty})$
for any $u\in D(G_{\infty})$ and
\begin{equation}
G_{\infty}(|u|^p) = pu|u|^{p-2}\G u+p(p-1)|u|^{p-2}\langle
Q\nabla_x u,\nabla_x u \rangle.
\label{form-Gfp}
\end{equation}
Recalling that $D(G_p)$ is continuously embedded in $W^{0,1}_{p}(\U
\times \R^d,\mu)$ by Proposition \ref{thm-2.7}, formula
\eqref{form-Gfp} implies that the nonlinear operator $u \mapsto |u
|^p$ is continuous from $D(G_{\infty})$ (endowed with the
$D(G_p)$-norm) to $D(G_1)$. Estimate $\|\,|u|^p\,\|_{D(G_1)} \leq
C_p\|u\|_{D(G_p)}^p$ follows  using the H\"older inequality in the
right-hand side of \eqref{form-Gfp}. Since $D(G_{\infty})$ is dense
in $D(G_p)$, the statement follows. \end{proof}

\begin{theorem}
\label{LogSob_dominio}
For every $p\geq 2$ and for every $u\in D(G_p)$, \eqref{logsob ineq mod} holds true.
\end{theorem}
\begin{proof} Fix $p \geq 2$ and   $u \in D(G_p)$. Then, there exists a
sequence $(u_n) \subset D(G_{\infty})$ such that $u_n \to u$ in the
graph norm of $G_p $. Possibly replacing  $(u_n)$ by a subsequence,
we may assume that $u_n \to u$ pointwise a.e. By Proposition
\ref{dsharp any}, for any $n \in \N$ we have
\begin{eqnarray*}
 \int_{\U \times \R^d } |u_n|^p \log (|u_n|^p)d \mu
  \leq
\frac{1}{T} \int_0^T \Pi |u_n|^p \log (\Pi |u_n|^p)ds +
\frac{p^2\Lambda}{2 |r_0|} \int_{\U \times \R^d } |u_n|^{p-2}
|\nabla_x u_n |^2 d\mu. \label{star-luca-2}
\end{eqnarray*}
As a first step, we prove  that
\begin{equation}
\lim_{n \to\infty} \int_0^T \Pi |u_n|^p \log (\Pi |u_n|^p)ds =
\int_0^T \Pi |u|^p \log (\Pi |u|^p)ds. \label{star-luca-2}
\end{equation}
By Lemma \ref{Le:nonlineare}, $|u_n|^p \to |u|^p$ in $D(G_1)$, as
$n\to\infty$. Therefore, $\lim_{n\to\infty} \Pi(|u_n|^p) = \Pi
(|u|^p)$  in $D(G_1)$. As we already mentioned at the beginning of
the section, the part of $G_1$ in $\Pi (L^1(\U \times\R^d, \mu)) $
is the time derivative $-D_s$ with domain isomorphic to
$W^{1,1}(\U,\frac{ds}{T})$.  It follows
that $\lim_{n\to\infty}\Pi(|u_n|^p) = \Pi (|u|^p)$ in
$W^{1,1}(\U,\frac{ds}{T})$ and, since $W^{1,1}((0,T),\frac{ds}{T})$
is continuously embedded in $L^{\infty}(0,T)$, and the function $y
\mapsto y \log y$ is $\alpha$-H\"older-continuous for any $\alpha
\in (0,1)$ on bounded sets of $[0,+\infty)$, we get
\begin{align*}
\frac{1}{T}  \int_0^T | \Pi |u_n|^p \log (\Pi |u_n|^p) -
\Pi |u|^p \log (\Pi |u|^p)| ds &\leq
\frac{C_1}{T}  \int_0^T | \Pi |u_n|^p - \Pi |u|^p |^{\alpha} ds\\
& \leq  C_1   \| \Pi |u_n|^p - \Pi |u|^p\|_{\infty}^{\alpha}\\
& \leq  C_2   \| \Pi |u_n|^p - \Pi
|u|^p\|^{\alpha}_{W^{1,1}((0,T),ds/T)},
\end{align*}
for some positive constants $C_i$ ($i=1,2$). Then,
\eqref{star-luca-2} follows.
Since the function $u\mapsto
H(u)=\int_{\U\times\R^d}|u|^{p-2}|\nabla_xu|^2d\mu$ is continuous in
$D(G_p)$ by  the proof of Corollary \ref{cor:2}(c), $H(u_n)$ tends
to $H(u)$ as $n\to +\infty$. Now, denote by $\log_-(y)$ and
$\log_+(y)$ the negative and the positive parts of $\log(y)$, i.e.,
\begin{eqnarray*}
\log_-(y):= \max\{0,-\log(y)\}, \quad \log_+(y):= \max\{0,
\log(y)\},\qquad\;\,y>0.
\end{eqnarray*}
Taking into account that the function $y \mapsto y^p \log _-(y^p)$
is Lipschitz continuous, we get
\begin{eqnarray*}
\int_{\U \times \R^d } \left| |u|^p \log_-(|u|^p) - |u_n|^p
\log_-(|u_n|^p) \right| d \mu \leq  C_3\int_{\U \times \R^d } \left
||u|^p-|u_n|^p\right | d \mu,
 \end{eqnarray*}
 for some constant $C_3>0$, and the  right-hand side tends to $0$ as $n \to\infty$.

By the  Fatou Lemma we have
\begin{align*}
&\int_{\U \times \R^d } |u|^p \log _+(|u|^p) d \mu \leq \liminf
_{n\to\infty}   \int_{\U \times \R^d } |u_n|^p \log _+(|u_n|^p) d
\mu
\\
\leq & \lim_{n \to\infty} \left( \int_{\U \times \R^d } |u_n|^p \log_-(|u_n|^p)d\mu
+ \frac{1}{T} \int_0^T \Pi |u_n|^p \log (\Pi |u_n|^p) ds\right.\\
&\qquad \qquad \left. +
\frac{p^2 \Lambda}{2 |r_0|} \int_{\U \times \R^d } |u_n|^{p-2}|\nabla_x u_n|^2 d \mu  \right)\\
=& \int_{\U \times \R^d } |u|^p \log_-(|u|^p) d \mu  +
\frac{1}{T}\int_0^T \Pi |u|^p \log (\Pi |u|^p) ds + \frac{p^2
\Lambda}{2 |r_0|} \int_{\U \times \R^d } |u|^{p-2} |\nabla_x u |^2 d
\mu,
\end{align*}
which implies \eqref{logsob ineq mod} and concludes the proof. \end{proof}


\subsection{Compactness in $L^p$ spaces}


If the domain $D(G_{p_0})$ is compactly embedded in $L^{p_0}(\U \times \R^d,\mu)$ for some $p_0$, a lot of nice consequences follow.

\begin{theorem}
\label{Th:cons} Under Hypothesis \ref{hyp1},  assume that the domain
of $G_{p_0}$ is compactly embedded in $L^{p_0}(\U \times \R^d,\mu)$
for some $p_0\in [1,\infty]$.
Then, for every $p\in (1,\infty)$ the domain of $G_{p }$ is
compactly embedded in $L^{p }(\U \times \R^d,\mu)$, and
\begin{itemize}
\item[(i)]
the spectrum of $G_p$ consists of isolated eigenvalues independent
of $p$, for $p\in (1,\infty)$. The associated spectral projections are independent of $p$, too;
\item[(ii)]
the growth bounds $\omega_p $ defined in \eqref{omegap,gammap} are
independent of $p\in (1,\infty)$. Denoting by $\omega_0$ their
common value, for every $p\in (1, \infty)$ we have
\begin{eqnarray*}
\omega_0=  \sup \,\{{\rm Re}\,\lambda: \lambda \in
\sigma(G_p)\setminus i\R \}.
\end{eqnarray*}
\end{itemize}
If in addition Hypothesis \ref{hyp2} too is satisfied, then
\begin{itemize}
\item[(iii)]
statement (a) of Corollary \ref{Cor:spectralgap} holds;
\item[(iv)]
 $\omega_0 <0$. Moreover,  for every $\omega > \omega_0$, $p\in (1,\infty)$
there exists $M>0$ such that
\begin{equation}
\label{a} \hskip 20pt\| P(t,s)\varphi - m_s \varphi\|_{L^p(\R^d,
\mu_t)} \leq M e^{\omega (t-s)}\| \varphi\|_{L^p(\R^d, \mu_s)},
\quad t>s, \; \varphi \in L^p(\R^d, \mu_s).
\end{equation}
\end{itemize}
\end{theorem}
\begin{proof} Suppose that $D(G_{p_0})$ is compactly embedded in
$L^{p_0}(\U\times\R^d,\mu)$. Then, for any $\lambda>0$ the resolvent
operator $u\mapsto \int_{0}^{\infty}e^{-\lambda t} \T(t)u\, dt$ is
compact in $L^{p_0}(\U \times \R^d,\mu)$, and since it is bounded in
all spaces $L^p(\U \times \R^d,\mu)$, $1\leq p \leq\infty$, it is
compact in all spaces $L^p(\U \times \R^d,\mu)$, $1 <p <\infty$, by
interpolation. See e.g.,  \cite[(proof of) Thm. 1.6.1]{davies}, for
$p_0<\infty$, and \cite[Prop. 4.6]{MPW} for $p_0=\infty$.
Since the domain of $G_p$ coincides with the range of $R(\lambda,
G_p)$, it is compactly embedded in $L^p(\U \times \R^d,\mu)$.

Let us now prove statements (i) to (iv).

\vspace{2mm}

(i). By the general spectral theory, the spectrum of $G_p$ consists
of isolated eigenvalues. Applying \cite[Cor. 1.6.2]{davies} to
the resolvent $R(\lambda, G_p)$ for a fixed $\lambda >0$, it follows
that the spectrum of $R(\lambda, G_p)$ is independent of $p$ and,
hence, the spectrum of $G_p$ is independent of $p$. It also follows that the spectral projections are independent of $p$.

\vspace{2mm}

(ii). Fix any $p\in (1,\infty)$  and denote by $G_{\Pi}$, $\T
_{\Pi}(t)$, respectively,  the parts of $G_p$, $\T(t)$ in $\Pi
(L^p(\U \times \R^d, \mu))$, and by $G_{I-\Pi}$, $\T _{I-\Pi}(t)$,
respectively, the parts of $G_p$, $\T(t)$ in $(I-\Pi )(L^p(\U \times
\R^d, \mu))$.

Since $\Pi$ commutes with $\T(t)$, then $\sigma (G_p) = \sigma
(G_{\Pi}) \cup \sigma(G_{I-\Pi})$. The   spectrum of $G_{\Pi} $ is
the set $\{ 2k \pi i/T: k\in \Z\}$, since $\Pi (L^p(\U \times \R^d,
\mu))$ is isometric to $L^{p}(\U, \frac{ds}{T})$ and $G_{\Pi} =
-D_s$ on $D(G_{\Pi}) = \Pi (D(G_p))$. Therefore,
\begin{eqnarray*}
\sup \{{\rm Re}\,\lambda: \lambda \in\sigma(G_p), \; {\rm
Re}\,\lambda <0\}  = \sup \{{\rm Re}\,\lambda: \lambda \in \sigma
(G_{I-\Pi})  \}.
\end{eqnarray*}
Let us prove that such suprema coincide with $\omega_p$. This will
imply that $\omega_p$ is independent of $p$, because the left-hand
supremum is independent of $p$.

To this aim, we remark that, although the operators $\T(t)$ are not
compact, $\sigma (\T(t))\setminus \{0\}$ consists of eigenvalues.
Indeed, by the Spectral Mapping Theorem \ref{Th:SMT}, $\sigma (\T(t)
)\setminus \{0\} = e^{t \sigma (G_p)}$, and by the general theory of
semigroups (e.g., \cite[Thm.~IV.3.7]{EN}) $P\sigma( \T(t) )\setminus
\{0\} = e^{t P\sigma (G_p)}$, where $P\sigma $ denotes the point
spectrum. Since $ \sigma (G_p) = P \sigma (G_p)$, then $\sigma
(\T(t)) \setminus \{0\}  = P\sigma (\T(t) )\setminus \{0\} $. As a
consequence,  also $\sigma (\T _{I-\Pi}(t)) \setminus \{0\} $
consists of eigenvalues, because the elements of $\sigma  (\T
_{I-\Pi}(t)) $, which are not eigenvalues,
 are  contained in $ \sigma (\T(t))\setminus P\sigma(\T(t))$, which
do not contain nonzero elements.  Again by the spectral mapping
theorem for the point spectrum, $\sigma (\T _{I-\Pi}(t)) \setminus
\{0\}  = e^{t \sigma(G_{I-\Pi})}$ i.e., the semigroup $\T
_{I-\Pi}(t)$ satisfies the spectral mapping theorem. This implies
that $\omega_p = \sup \{{\rm Re}\,\lambda: \lambda \in \sigma
(G_{I-\Pi})  \}$, because $\omega_p$ coincides with the logarithm of
the spectral radius of $\T_{I-\Pi} (1)$ (e.g., \cite[Prop.~IV.2.2]{EN}).

\vspace{2mm} (iii). Since $\T_{I-\Pi}$ is strongly stable by Theorem
\ref{decadimento},  $G_{I-\Pi}$ cannot have eigenvalues on the
imaginary axis. Therefore, $i\R $  is contained  the resolvent set
of $G_{I-\Pi}$. The arguments used in the proof of the statement (a)
of Corollary \ref{Cor:spectralgap}  yield the statement.

\vspace{2mm}
(iv).  We already remarked that $\sigma (G_{I-\Pi})\cap i\R = \varnothing$. Consequently, the spectrum of $\T
_{I-\Pi}(1)$ does not intersect the unit circle. It follows that
\begin{equation}
\label{spettroT(1)} \sup \{ |\zeta|: \zeta \in \sigma (\T
_{I-\Pi}(1))  \} <1.
\end{equation}
Indeed, if there were a sequence of eigenvalues $(\zeta_n)$
of $\T_{I-\Pi}(1)$ such that $\lim_{n\to\infty} |\zeta_n| =1$,   a
subsequence would converge  to an element $\zeta $ with modulus $1$,
and since the spectrum is closed,   $\zeta \in \sigma (\T
_{I-\Pi}(1))$. But this is impossible. Hence, \eqref{spettroT(1)}
holds.

It follows that  there exists $a<1$ such that
\begin{eqnarray*}
\|\T_{I-\Pi} (n)\|_{{\mathcal L}(L^p(\U\times\R^d,\mu))} = \|
(\T_{I-\Pi} (1))^n\|_{{\mathcal L}(L^p(\U\times\R^d,\mu))} \leq a^n,
\end{eqnarray*}
for $n$ large, and since
\begin{align*}
\|\T_{I-\Pi} (t)\|_{{\mathcal L}(L^p(\U\times\R^d,\mu))}  & =\|\T_{I-\Pi} (t-n)
\T_{I-\Pi} (n)\|_{{\mathcal L}(L^p(\U\times\R^d,\mu))} \\
&
  \leq \| \T_{I-\Pi}(n)\|_{{\mathcal L}(L^p(\U\times\R^d,\mu))},
  \end{align*}
for $n\leq t <n+1$, $\|\T_{I-\Pi} (t)\|_{{\mathcal L}(L^p(\U\times\R^d,\mu))}$ decays
exponentially as $ t\to \infty$, i.e., $\omega_p <0$.

Estimate \eqref{a}  follows from Theorem
\ref{Pr:EquivExp}.
\end{proof}

As in the autonomous case,   log-Sobolev inequalities imply that
$D(G_p)$ is compactly embedded in $L^p(\U \times \R^d,\mu)$, for
every $p\in (1,\infty)$.

\begin{theorem} \label{Th:comp1} Let Hypotheses  \ref{hyp1}  and  \ref{hyp2} hold.
Assume that  \eqref{logsob ineq mod}
holds for $p=2$ and for every $f\in D(G_2)$. Then, for any $p\in
(1,\infty)$, $D(G_p)$ is compactly embedded in $L^p(\U \times
\R^d,\mu)$.
\end{theorem}
\begin{proof}
By Theorem \ref{Th:cons}, it is enough to prove that $D(G_2)$ is compactly embedded in $L^2(\U \times \R^d, \mu)$.

We shall show that, for every $\eps >0$, the unit ball $B$ of $D(G_2)$
may be covered by a finite number of balls of $L^2(\U \times \R^d,\mu)$
with radius not greater than $\eps$.

Fix $u\in B$, $k>1$ and set $E := \{|u| < k\}$. For every $R>0$ we have, by \eqref{logsob ineq mod},
\begin{align*}
\int_{0}^{T}\int_{ B(0,R)^c} u^2 d \mu \leq &
\int_{0}^{T}\int_{ B(0,R)^c}  \one_E k^2 d \mu  + \frac{1}{\log (k^2)}\int_{0}^{T}\int_{ B(0,R)^c}
\one_{E^c} u^2 \log (u^2) d \mu \nonumber\\
\leq &\frac{k^2}{ T} \int_0^T ds \int_{B(0,R)^c} d\mu_{s} \nonumber \\
& + \frac{ 1}{\log(k^2)}\bigg(\frac{1}{T} \int_0^T \Pi(u^2)   \log
(\Pi u^2) ds + C_1\int_{\U\times\R^d}|\nabla_x u|^2 d\mu \bigg),
\end{align*}
for some positive constant $C_1$, independent of $u$ and $k$.

Fix $\varepsilon >0$. By Lemma \ref{Le:nonlineare}, $u^2 \in
D(G_1)$, and $\|u^2\|_{D(G_1)} \leq C_2\|u\|_{D(G_2)}^2 \leq C_2$, with $C_2$ independent of $u$.
Therefore, $\Pi u^2 $ belongs to the domain of the part of $G_1$ in
$\Pi (L^1(\U \times \R^d,\mu))$, which is isomorphic to $W^{1,1}(\U,
\frac{ds}{T})$. By the Sobolev embedding for the Lebesgue measure,
$\|\Pi u^2\|_{L^{\infty}(0,T)}$ is bounded by a constant independent
of $u$, so that $\int_0^T \Pi u^2   \log (\Pi u^2) ds$ is bounded by
a constant independent of $u$. Also the integral
$\int_{\R^{1+d}}|\nabla_x u|^2 d\mu$ is bounded by a constant
independent of $u$, by Proposition \ref{thm-2.7}. So, there exists $M>0$
such that $\frac{1}{T} \int_0^T \Pi u^2   \log (\Pi u^2) ds+ C_1
\int_{\U\times\R^d}|\nabla_x u|^2 d\mu\leq M$, for every $u\in B $.
Taking $k $ large enough, we get
\begin{eqnarray*}
\frac{ 1}{\log(k^2)}\left(\frac{1}{T}\int_0^T \Pi u^2 \log (\Pi u^2)
ds + C_1 \int_{\U\times\R^d}|\nabla_x u|^2 d\mu \right) \leq \frac{\eps}{2}.
\end{eqnarray*}
By Theorem \ref{Th:P(t,s)}(v) the measures $\mu_s$ are tight, so
that there exists $R>0$ such that $\frac{k^2}{ T} \int_0^T ds
\int_{B(0,R)^c} d\mu_{s} \leq \eps/2$. Summing up,
\begin{eqnarray*}
\int_{(0,T)\times B(0,R)^c} u^2 d \mu  \leq \eps.
 \end{eqnarray*}

By Corollary \ref{restrizione}, $D(G_2)$ is contained in
$W^{1,2}_{2,\loc}(\U \times \R^d, ds\times dx)$ and the restriction
operator ${\mathcal R}:D(G_2)\to W^{1,2}_2(\U \times B(0,R), ds\times dx)$,
${\mathcal R}u=u_{|\U \times B(0,R)}$, is continuous. Since the
embedding of $W^{1,2}_2(\U \times B(0,R), ds\times dx)$ in $L^2(\U
\times B(0,R),ds\times dx)$ is compact, there exist $f_1,\dots,f_k
\in L^2(\U \times B(0,R),ds\times dx)$ such that the balls
$B(f_i,\varepsilon)$ cover the restrictions of the functions of $B$
to $\U \times B(0,R)$. Let $\tilde{f}_i$ denote the null extension
of $f_i$ to $\U \times \R^d$. Then $B \subset \bigcup_{i=1}^k
B(\tilde{f}_i,2 \varepsilon)$, and the statement follows.
\end{proof}

\begin{remark}
\label{Rem:confronto}
{\em Under Hypotheses  \ref{hyp1}  and  \ref{hyp2}, if the diffusion coefficients are independent of $x$  and
$r_0<0$, then the assumptions of Theorem \ref{Th:comp1} are
satisfied, hence all the statements of Theorem \ref{Th:cons} hold, as well as the statements of Corollaries \ref{Cor:spectralgap} and
\ref{Cor:ca}. Since $r_0 = \ell_2 = \omega_0$, statement (ii) of Theorem \ref{Th:cons} is sharper than the statements of Corollaries \ref{Cor:spectralgap} and
\ref{Cor:ca} for $1<p<2$, while estimate \eqref{eq2} is sharper than statement (iv)  of Theorem \ref{Th:cons} for $p\geq 2$. }
\end{remark}

\section{Examples}
\setcounter{equation}{0}

\subsection{Time dependent Ornstein-Uhlenbeck operators}

Let us consider the operators
\begin{eqnarray*}
(\A(t)\varphi)(x)=\frac12 \tr  \left( B(t)B^*(t)D_x^2\varphi   (x)\right)+ \langle A(t)x+f(t),\nabla \varphi(x)\rangle ,\quad x\in\R^d,
\end{eqnarray*}
with continuous and $T$-periodic  data $A,B:\R\to{\mathcal L}(\R^d)$ and $f:\R\to\R^d$.
The ellipticity condition \eqref{ellipticity} is satisfied provided $\det B(t)\neq 0$ for every $t\in \R$.

In  \cite{GL2} asymptotic behavior results for the backward
evolution operator $P(s,t)$, $s\leq t$, associated to the family
$\{\A(t)\}$ in $L^2$ spaces have been proved, as well as  spectral
properties of the parabolic operator $u\mapsto \A(s)u +D_s u $.
Here, we  consider forward evolution operators $P(t,s)$, $t\ge s$,
and the parabolic operator $u\mapsto \A(s)u -D_s u $. Reverting
time, there is no difficulty to pass from backward to forward.

Let $U(t,s)$ be the evolution operator in $\R^d$,  solution of
$\frac{\partial}{\partial t} U(t,s) = A(t) U(t,s)$, $U(s,s) = I$. A
(unique) $T$-periodic evolution system of measures $\{ \mu_s: s\in
\R\}$ exists provided the growth bound $\omega_0(U)$ of  $U(t,s)$ is
negative; in this case the measures $\mu_t$ are explicit Gaussian
measures.

The results of this paper allow to extend most of the $L^2$ asymptotic behavior results of \cite{GL2} to the $L^p$ setting, with $p\in (1, \infty)$.
In fact, the log-Sobolev inequality  \eqref{logsob ineq mod}
holds for $p=2$, for every $u\in D(G_2)$. It was proved in \cite{DPL} for every $u\in C^{1,2}_{b}(\U\times \R^d)$ which is dense in $D(G_2)$, and the procedure of Theorem \ref{LogSob_dominio} allows to extend it to all the functions $u\in D(G_2)$.
Moreover, Proposition 2.4 of \cite{GL2} shows that $\omega_2 =\omega_0(U)$.
Therefore, all the statements of Theorem \ref{Th:cons} hold,  with $\omega_0 =\omega_0(U)$.

Note that our assumption of H\"older regularity of the coefficients
is not needed here,  because the proof of Theorem
\ref{LogSob_dominio} is independent of time regularity of the
coefficients.

\subsection{Diffusion coefficients independent of $x$}

Let now consider the operators $\A (t)$ defined in \eqref{A(t)} with $T$-periodic diffusion coefficients depending only on time, under
the regularity and ellipticity assumptions of Hypothesis  \ref{hyp1}(i)--(ii). For every $n\in \N$ the function  $V(x):=1+|x|^{2n}$  satisfies
Hypothesis  \ref{hyp1}(iii) provided that there exists $R>0$ such that
\begin{eqnarray*}
\sup_{s\in \R, \,|x|\geq R} \frac{\langle b(s,x) ,x\rangle}{|x|^2} <0 .
\end{eqnarray*}
In this case the statements of  Theorem \ref{Th:P(t,s)} and of
Proposition \ref{Pr:uniqueness} hold. So, there exists a Markov
evolution operator $P(t,s)$ with a unique $T-$periodic evolution
system of measures $\{\mu_s: s\in \R\}$. The measures $\mu_s$ have
uniformly bounded moments of every order, i.e.,
\begin{eqnarray*}
\sup_{s\in \R} \int_{\R^d} |x|^k \mu_s(dx) <\infty, \quad k\in \N.
\end{eqnarray*}

If moreover  the derivatives $D_ib_j$ belong to $C^{\alpha/2,\alpha}_{\loc}(\U \times \R^d)$ and there exists $r_0\in \R$ such that
\begin{eqnarray*}\qquad\;\;\;\;\;\;\;\langle \nabla_x b(s,x)\xi,\xi \rangle \leq r_0|\xi|^2,
\quad   (s,x) \in \U \times \R^d ,\;\,   \xi \in \R^d,
\end{eqnarray*}
then Hypothesis  \ref{hyp2} holds too.

Applying Theorem \ref{decadimento}, statements (ii) to (iv) of Theorem \ref{Pr:equivalenza} hold.

In the case that $r_0<0$, we have $\ell_2 =r_0<0$, and the  log-Sobolev
inequalities of Subsection 3.1 hold. By Theorem \ref{Th:comp1} the
domain $D(G_p)$ is compactly embedded in $L^p(\U\times \R^d, \mu)$
for   $p\in (1, \infty)$ and all the statements of Theorem
\ref{Th:cons} hold. Moreover the statements   of Corollaries
\ref{Cor:spectralgap} and \ref{Cor:ca} hold. See Remark
\ref{Rem:confronto}.

\subsection{General diffusion coefficients}

In the general case, setting again $V(x):=1+|x|^{2n}$, we have
\begin{eqnarray*}
\A (s) V(x) = 2n |x|^{2n} \bigg [ (2n-2) \frac{ \langle Q(s,x)x,x\rangle }{|x|^4} +  \frac{\tr \;Q(s,x)}{|x|^2} + \frac{\langle b(s,x), x\rangle}{|x|^2} \bigg],
\end{eqnarray*}
for any $(s,x)\in \R^{1+d}$, so that Hypothesis  \ref{hyp1}(iii) is satisfied by $V$ provided   there exists $R>0$  such that
\begin{equation}
\label{eq:suffV}
\sup_{s\in \R, \,|x|\geq R} (2n-2+d)\frac{ \Lambda(s,x)}{|x|^2} +   \frac{\langle b(s,x) ,x\rangle}{|x|^2} <0,
\end{equation}
where $\Lambda(s,x)$ is the greatest eigenvalue of $Q(s,x)$. If also
the regularity and ellipticity assumptions of  Hypothesis
\ref{hyp1}(i)--(ii) are satisfied, by Theorem \ref{Th:P(t,s)} and
Proposition \ref{Pr:uniqueness}   there exists a Markov evolution
operator $P(t,s)$ with a unique $T-$periodic evolution system of
measures $\{\mu_s: s\in \R\}$; the measures $\mu_s$ satisfy
\begin{eqnarray*}
\sup_{s\in \R} \int_{\R^d} |x|^{2n} \mu_s(dx) <\infty.
\end{eqnarray*}
If, in addition,  Hypothesis  \ref{hyp2} is satisfied and   there exists $C>0$ such that
\begin{eqnarray*}
\|Q(s,x)\|_{{\mathcal L}(\R^d)} \leq C(1+|x|)^{2n+1}, \quad s\in \R, \;x\in \R^d,
\end{eqnarray*}
then the assumptions of Proposition \ref{DecadimentoGradiente} hold. Indeed,  since $\Lambda (s,x)>0$,    \eqref{eq:suffV} implies that  $\langle b(s,x) ,x\rangle <0$ for $|x|\geq R$ and $s\in \R$, so that the second condition of \eqref{stimapiu'} is satisfied.
Then,  Theorem   \ref{decadimento} yields that statements (ii) to (iv) of Theorem \ref{Pr:equivalenza} hold.

If in addition the diffusion coefficients are bounded and the number $\ell_2$ in \eqref{ell_p} is negative, all the assumptions of Corollaries
\ref{Cor:spectralgap}  and \ref{Cor:ca} are satisfied,   we have the exponential decay rates given by Corollary  \ref{Cor:ca} and the spectral properties of the operators $G_p$ given by Corollary \ref{Cor:spectralgap}.


\begin{thebibliography}{999}



\bibitem{many}
C. An\'e, S. Blach\`ere, D. Chafa\"i, P. Foug\`eres, I. Gentil, F. Malrieu, C. Roberto, G. Scheffer:  Sur les in\'egalit\'es de Sobolev logarithmiques, Collection ``Panoramas et Synth\`eses", SMF 2000, no. 10.




\bibitem{BKR} V.I.~Bogachev, N.V.~Krylov, M.~R\"ockner: \emph{On
regularity of transition probabilities and invariant measures of
singular diffusion under minimal conditions}, Comm. Partial
Differential Equations \textbf{26} (2001), no. 11--12,
2037--2080.


\bibitem{C}
S.~Cerrai: Second order PDE's in finite and infinite dimensions. A probabilistic approach. Lecture Notes in Math.  \textbf{1762}, Springer-Verlag, Berlin, 2001.


\bibitem{CL}
C.~Chicone, Y.~Latushkin: Evolution Semigroups in
 Dynamical Systems   and Differential Equations. Amer. Math. Soc., Providence (RI), 1999.

\bibitem{DD} G.~Da Prato, A.~Debussche:  \emph{2D Stochastic Navier-Stokes
equations with a time-periodic forcing
term},  J. Dynam. Differential Equations \textbf{20} (2008), no. 2, 301--335.

\bibitem{DPG}
G.~Da Prato,~B. Goldys:   \emph{Elliptic operators on $\R^d$ with unbounded coefficients}, J. Diff. Eq. \textbf{172} (2001),   333--358.

\bibitem{DPL}
G.~Da Prato,~A. Lunardi: \emph{Ornstein-Uhlenbeck operators with
time periodic coefficients}, J. Evol. Equ. \textbf{7} (2007), no. 4,  587--614.

  \bibitem{DPR}
  G.~Da Prato, M.~R\"ockner: \emph{Dissipative stochastic equations in
  Hilbert space with time dependent coefficients}, Rend. Lincei Mat.
  Appl. \textbf{17} (2006),  397--403.

 \bibitem{DPZ}
G.~Da Prato, J.~Zabczyk: Ergodicity for infinite-dimensional systems. London Mathematical Society Lecture Note Series, 229. Cambridge University Press, Cambridge, 1996.

 \bibitem{DPZbrutto}
 G.~Da Prato, J.~Zabczyk: Second Order Partial Differential Equations in Hilbert Spaces. London Mathematical Society Lecture Note Series, 293. Cambridge University Press, Cambridge, 2002.

\bibitem{davies}
E.B.~Davies: Heat Kernels and Spectral Theory, Cambridge Univ.\ Press (1989).

\bibitem{DS}  J. D.~Deuschel, D.~Stroock:  Large Deviations. Academic Press, San Diego, 1984.

\bibitem{EN} K.J.~Engel, R.~Nagel: One-parameter semigroups for linear
evolution equations. Graduate Texts in Mathematics \textbf{194},
Springer-Verlag, New York, 2000.

\bibitem{EK} S.N.~Ethier, T.~Kurtz: Markov Processes. Characterization and Convergence. Wyley, New York, 1986.

\bibitem{friedman}
A.~Friedman, Partial differential equations of parabolic type. Prentice Hall, 1964.

\bibitem{GLS}
M.~Geissert, L.~Lorenzi,   R.~Schnaubelt: \emph{$L^p$ regularity for
parabolic operators with unbounded time-dependent coefficients},
Annali Mat. Pura Appl. (to appear).

\bibitem{GL1} M.~Geissert,  A.~Lunardi: \emph{Invariant measures and maximal $L^2$
regularity for nonautonomous Ornstein-Uhlenbeck equations},  J.
Lond. Math. Soc. (2)  \textbf{77}  (2008), no. 3,  719--740.

\bibitem{GL2} M.~Geissert,  A.~Lunardi:
\emph{Asymptotic behavior and hypercontractivity in nonautonomoous
Ornstein-Uhlenbeck equations}, J. Lond.\ Math. Soc. (2) \textbf{79}
(2009),  85--106.

\bibitem{Gross} L.~Gross: \emph{Hypercontractivity, Logarithmic Sobolev Inequalites, and Applications: A Survey of Syrveys}, in: Diffusion, quantum theory, and radically  elementary mathematics, Math. Notes, 47, Princeton Univ. Press, Princeton, NJ, 2006, 45--73.

\bibitem{KLL} M.~Kunze, L.~Lorenzi,   A.~Lunardi:
\emph{Nonautonomous Kolmogorov parabolic equations with unbounded
coefficients}. Trans. Amer. Math. Soc. (to appear).


\bibitem{LZ} L.~Lorenzi   A.~Zamboni: \emph{Cores and regularity
for nonautonomous operators with unbounded coefficients} J.
Differential Equations {\bf 246} (2009),  2724--2761.


\bibitem{MPP} G.~Metafune, D.~Pallara, E.~Priola: \emph{Spectrum of Ornstein-Uhlenbeck operators in $L^p$ spaces with respect to invariant measures}. J. Funct. Anal. {\bf 196} (2002), no. 1, 40--60.


\bibitem{MPW}
G.~Metafune, D.~Pallara, M.~Wacker: \emph{Compactness properties of Feller semigroups},  Studia Math. 153 (2002), no. 2, 179--206.

\bibitem{SV}
D.W.~Strook, S.R.S.~Varadhan: Multidimensional diffusion processes, Classics in Mathematics.
Springer-Verlag, Berlin 2006.

\bibitem{Triebel}
H.~Triebel:  Interpolation Theory, Function Spaces,
Differential
Operators, North-Holland, Amsterdam, 1978.

\bibitem{Wang} F.-Y.~Wang: \emph{A character of the gradient estimate for diffusion semigroups}, Proc. A.M.S. {\bf 133} (2004), no. 3, 827--834.


\end{thebibliography}
\end{document}